\documentclass[a4paper,12pt,oneside]{article}
\usepackage{amsfonts}
\usepackage[a4paper,hmargin={2.54cm,2.54cm},vmargin={3.17cm,3.17cm}]{geometry}
\usepackage{amsmath,amssymb,amsthm,amsfonts}
\usepackage{bbm}
\numberwithin{equation}{section}

\usepackage{abstract}
\usepackage{graphicx}
\newtheoremstyle{mystyle}{12pt}{12pt}{\itshape}{0cm}{\bfseries}{}{1em}{}
\theoremstyle{mystyle}

\newtheorem{theorem}{Theorem}[section]

\newtheorem{lemma}{Lemma}[section]
\newtheorem{proposition}{Proposition}[section]
\newtheorem{corollary}{Corollary}[section]

\newtheorem{remark}{Remark}[section]

\begin{document}

\title{\bf{Smallest Gaps Between Eigenvalues of Random Matrices With Complex Ginibre,  Wishart and Universal Unitary Ensembles}}
 \author{Dai Shi\footnote{Institute of Computational and Mathematical Engineering, Stanford University, Stanford, 94305}, Yunjiang Jiang\footnote{Mathematics Department, Stanford University, Stanford, 94305}}

\date{}

\renewcommand{\thepage}{\roman{page}}
\setcounter{page}{1} 
\maketitle\renewcommand{\thepage}{\arabic{page}}\setcounter{page}{1}

\begin{abstract}
\noindent In this paper we study the limiting distribution of the $k$ smallest gaps between eigenvalues of three kinds of random matrices -- the Ginibre ensemble, the Wishart ensemble and the universal unitary ensemble. All of them follow a Poissonian ansatz. More precisely, for the Ginibre ensemble we have a global result in which the $k$-th smallest gap has typical length $n^{-3/4}$ with density $x^{4k-1}e^{-x^4}$ after normalization. For the Wishart and the universal unitary ensemble, it has typical length $n^{-4/3}$ and has density $x^{3k-1}e^{-x^3}$ after normalization. \\
\end{abstract}

\section{Introduction}
The goal of this paper is to establish the limiting distribution of the $k$ smallest gaps between eigenvalues. In particular, we focus on the following three ensembles: the Ginibre ensemble (where the entires are i.i.d. complex Gaussian), the Wishart ensemble (where it is the sample covariance matrix of $n$ independent complex Gaussian vectors), and the universal unitary ensemble (where the matrix is invariant under unitary transformations with a general potential). More precisely, we address two problems here: pick an $n\times n$ matrix from a given ensemble, what is the typical length of the $k$ smallest eigenvalue gaps? As a second order approximation, under the correct scaling, what is the joint limiting distribution of these gaps as we send $n$ to infinity?

Through many years typical spacing between eigenvalues of random matrices has attracted a lot of attention, this is partially due to Montgomery's conjecture: the normalized gaps between zeros from a Riemann zeta function appear to share the same statistics of the normalized eigenvalue gaps of a circular unitary ensemble (CUE) (see \cite{13} for a review). We already know that in the Gaussian unitary ensemble (GUE) case, the expectation of the empirical distribution of the normalized gaps has the limiting density which can be expressed as a Fredholm determinant.  $p(s) := \partial^2_{ss}\det(I-\mathcal{K})_{L^2(0, s)}$ where $\mathcal{K}$ is the integral operator on $L^2(0, s)$ with the kernel $K(x, y) = \sin[\pi(x-y)]/[\pi(x-y)].$ The asymptotics of $p(s)$ when $s\to0$ and $s\to\infty$ have also been obtained (see Chapter 6 and Appendix 13 of \cite{6}). Later Kats and Sarnak in \cite{14} proved that this can also be extended to the CUE case and the corresponding orthogonal and symplectic case, with convergence rates provided. Similar results for universal unitary ensembles (UUE) are proved in \cite{15}. For the second order asymptotics, Soshnikov established in \cite{16} the Gaussian fluctuation around the limiting distribution for the CUE case. Recently, Tao and Vu  proved the \emph{Four Moment Theorem} in \cite{17} saying that local eigenvalue statistics are universal among all Wigner matrices provided that the first four moments are matched. Yau et al also proved the similar universality result (under different assumptions) by analyzing the time-reversal approximation of the Dyson Brownian motion. See \cite{18} and \cite{19}.

However, little was known about the extreme spacings, even in some simple ensembles. Vinson in \cite{4} first studied this subject and obtained results of smallest spacings for the CUE and GUE case by asymptotic moment estimation. More precisely, by matching the moments, he proved that the number of eigenvalue gaps in $(0, n^{-4/3}s)$ tends to the Poisson distribution. Later Soshnikov in \cite{5} established similar results for general determinantal random point processes on the real line with a  translation invariant kernel. Although the main idea is also moment (more precisely, cumulant) matching, he used another technique: he introduced the $s$\emph{-modified} point process and investigated the correlation function of it instead of the original process. Recently Ben Arous and Bourgade in \cite{1} developed Soshnikov's methods to solve the problem for CUE and GUE completely, showing that the joint distribution of the $k$ smallest gaps has a Poissonian limit. For other ensembles, there are almost no results concerning these. In this paper, we extend Ben Arous and Bourgade's method to analyze the smallest eigenvalue gaps of the Ginibre ensemble, the Wishart ensemble and the general universal unitary ensemble. 

For the Ginibre ensemble, we have two major issues. First, in the CUE or GUE case considered in \cite{1}, the eigenvalues (or eigenangles) are real. Hence we have a natural definition for the gap, i.e. $\lambda_{i+1}-\lambda_i$, thanks to the linear ordering in $\mathbb{R}$. In our case, however, the matrix is not hermitian so we have to deal with complex eigenvalues. Of course for any eigenvalue $\lambda_i$ it is natural to define $\min_{j\neq i}|\lambda_j-\lambda_i|$ to be the gap. However the point process $\sum_{i=1}^{n}\delta_{n^{3/4}\min_{j\neq i}|\lambda_j-\lambda_i|}$ cannot have a Poissionian limit because there are at least two identical points. The key problem here is that the search direction must be a half plane --- you can not ``look back'' when searching for the gap (recall in the real case we define the gap to be $\lambda_{i+1}-\lambda_i$ rather than $\min\{\lambda_{i} -\lambda_{i-1}, \lambda_{i+1}-\lambda_i\}$). Thus our remedy is to introduce a total ordering in $\mathbb{C}$ and then consider the modified point process $\sum_{i=1}^{n-1}\delta_{n^{3/4}\min_{j\neq i}\{|\lambda_j-\lambda_i|, \lambda_j\succeq\lambda_i\}}$ instead. We prove that the modified process has a Poisson limit. See Theorem \ref{thm1}.

However, the total ordering appeared in Theorem \ref{thm1} seems cumbersome, and our main concern is the joint distribution of the $k$ smallest elements $t_1, \ldots, t_k$ of the set $\{|\lambda_j-\lambda_i|: i\neq j\}$ --- with no ordering involved at all! We prove in Corollary \ref{col1} that for the small gaps the ordering introduced above will have hardly any effect. More rigorously, denote $\widetilde{t}_1, \ldots, \widetilde{t}_k$ to be the $k$ smallest elements of $\{|\lambda_j-\lambda_i|: \lambda_j\succ \lambda_i\}$, then we prove that $(t_1, \ldots, t_k) = (\widetilde{t}_1, \ldots, \widetilde{t}_k)$ with high probability, enabling us to analyze the gaps of the natural version from the ``ordered version''.

To state the second issue, we start from the GUE case. From the semi-circle law, the empirical measure of the eigenvalues has a semi-circle density which is supported on $[-2, 2]$. For any $\epsilon>0$ we refer $(-2+\epsilon, 2-\epsilon)$ as the \emph{bulk} of the eigenvalues and $(-2-\epsilon, -2+\epsilon)\cup(2-\epsilon, 2+\epsilon)$ as the \emph{edge}. Ben Arous in \cite{1} stated the result only for gaps in the bulk. This is due to a technical issue of the Plancerel-Rotach asymptotics of Hermite polynomials. For our Ginibre case, things are similar. Due to the circular law, this time the limiting measure is a uniform distribution on the unit circle in $\mathbb{C}$. The kernel function $K(x, x) = \pi^{-1}e^{-n|x|^2}\sum_{\ell = 0}^{n-1}(n|x|^2)^\ell/\ell!$ converges to $\pi^{-1}$ if $|x|<1$ and to $0$ if $|x|>1$. But the convergence rate is slow for $x$ near the boundary. Nevertheless, our result for the convergence of the smallest gaps is \emph{global}, that is, we do not restrict the gaps to be within the bulk of the unit circle --- they are freed not only to the edge, but further to the whole complex plane. This is done by a more careful analysis of the convergence rate of the kernel function, as is stated in Lemma \ref{good}. Meanwhile, in the following lemmas we analyzed separately the inner part, the boundary part and the outer part of the unit circle to eliminate the edge issues. In our situation, the orthogonal polynomial in the complex plane is just $z^n$. Compared to the Hermite polynomial in the GUE case, this is much simpler, which enables us to perform a more precise analysis. 

It's generally accepted (for example see \cite{6}) that the eigenvalue gaps are anti-correlated with each other, that is, large eigenvalue gaps tend to be followed with small ones and vice-versa. However, our result that the small eigenvalue gaps follows a Poissonian ansatz seems to imply that they are asymptotically independent. This has a natural explanation. As the gaps are anti-correlated, the small gaps are very unlikely to cluster together. That is, for the real case (such as GUE), eigenvalues in the set $\{\lambda_i: \lambda_{i+1}-\lambda_i< \text{ const }\cdot n^{-4/3}\}$ must fall apart. the distance between these $\lambda_i$'s are large enough so that these small gaps $\lambda_{i+1}-\lambda_i$  are nearly independent.
%

The rest of the paper is organized as follows. In subsection 1.1, 1.2 and 1.3, we will state the problem and our result, respectively, for the three ensembles. In section 2, we will provide a general overview of our methods. 
The proofs for the three ensembles will be given in section 3 to section 5, respectively. In section 6, a conclusion of the paper will be given.
\subsection{The Complex Ginibre Ensemble}
Let $X_n = (x_{ij})_{i, j=1}^n\in\mathbb{C}^{n\times n}$ be a random matrix with independent standard complex Gaussian entries. That is, 
\[
x_{ij} = \frac{u_{ij} + \sqrt{-1}v_{ij}}{\sqrt{2}}
\]
where $u_{ij}, v_{ij}$ are independent real Gaussian random variables with mean zero and variance one. Here we normalize each entry by the factor $\sqrt{2}$ to keep the variance unity. Such $X_n$ is called \emph{Complex Ginibre Ensemble}. For further discussion about this ensemble, see \cite{6}. 

Define $A_n = X_n/\sqrt{n}$ as the normalization $X_n$. Furthermore define $\lambda_1, \ldots, \lambda_n$ as the $n$ eigenvalues of $A_n$. By the circular law (see \cite{7}), the empirical spectrum distribution will almost surely converge to the uniform distribution of the unit disk in the complex plane. 

Since we are dealing with \emph{complex} eigenvalues now, we need to define a total ordering in $\mathbb{C}$ first. For two complex numbers $z_1$ and $z_2$, we define
\begin{equation}\label{s1.1}
z_1 \prec z_2 \iff \left\{\begin{array}{l}\Im(z_1)<\Im(z_2), \text{ or }\\ \Im(z_1) = \Im(z_2) \text{ and }\Re(z_1)<\Re(z_2) \end{array}\right.
\end{equation}
It is obvious that $(\cdot\prec\cdot)$ is just the standard lexicographical ordering.

Without losing generality, we can assume $\lambda_1\preceq\lambda_2\preceq\cdots\preceq\lambda_n$. 
Consider the point process on $\mathbb{R}^+\times\mathbb{C}$
\begin{equation}\label{1.1}
\chi^{(n)} = \sum_{i=1}^{n-1}\delta_{(n^{3/4}|\lambda_{i^*}-\lambda_i|, \lambda_i)}
\end{equation}
where $i^* = \arg\min_{j = 1 , j\neq i}^n\{|\lambda_j-\lambda_i|: \lambda_j\succeq\lambda_i\}$ is the index such that $\lambda_{i^*}$ is the closest one to $\lambda_i$ among all the eigenvalues greater than or equal to $\lambda_i$. 

The main result is about the convergence of $\chi^{(n)}$ to a Poisson point process. 

\begin{theorem}\label{thm1}
As $n\to\infty$, the process $\chi^{(n)}$ converges weakly to a Poisson point process $\chi$ in $\mathbb{R}^+\times\mathbb{C}$ with intensity 
\[
\mathbb{E}\chi(A\times I) = \biggl(\int_B|u|^2du\biggl)\biggl(\int_{I\cap\mathrm{D}(0, 1)}\frac{dv}{\pi^2}\biggl) = \frac{|I\cap\mathrm{D}(0, 1)|}{\pi}\int_Ar^3dr.
\]
for any bounded Borel sets $A\subset\mathbb{R}^+$ and $I\subset \mathbb{C}$. Here $\mathrm{D}(z, r)$ is a disk centered at $z$ and having radius $r$. Moreover $B = \{u\in\mathbb{C}: |u|\in A, u\succeq 0\}$ and $|\cdot|$ denotes the Lebesgue measure of the set.
\end{theorem}

\begin{remark}
In the definition of $i^* = \arg\min_{j = 1 , j\neq i}^n\{|\lambda_j-\lambda_i|: \lambda_j\succeq\lambda_i\}$, the constraint $\lambda_j\succeq\lambda_i$ is necessarily. Suppose $\lambda_{i_1}$ and $\lambda_{i_2}$ are the two closest eigenvalues. If we didn't impose the constraint, we would simultaneously have $i_1^* = i_2, i_2^* = i_1$ and hence $|\lambda_{i_0}-\lambda_{i_0^*}| = |\lambda_{i_1}-\lambda_{i_1^*}|$. This imply that at least two points in $\chi^{(n)}$ share the same first coordinate. Thus $\chi^{(n)}$ cannot have the Poissonian limit. 
\end{remark}

From Theorem \ref{thm1}, the intensity is proportional to $|I\cap\mathrm{D}(0, 1)|$. This is intuitive because by the circular law, the limiting empirical distribution is uniform inside the unit disk and the region outside $\mathrm{D}(0, 1)$ does not contribute to the intensity. 

As a comparison to non-repulsive i.i.d. samples, we define $\phi_i, 1\leq i\leq n$ to be independently and uniformly distributed in the unit circle. We can define a similar point process
\[
\widehat{\chi}^{(n)} = \sum_{i=1}^n\delta_{(n|\phi_{i^*}-\phi_i|, \phi_i)}.
\]
where $i^*$ are similarly defined as above. We have the similar version of Theorem \ref{thm1}.
\begin{theorem}\label{thm2.1}
As $n\to\infty$, the process $\widehat{\chi}^{(n)}$ converges weakly to a Poisson point process $\widehat{\chi}$ with intensity
\[
\mathbb{E}\widehat{\chi}(A\times I) = \frac{|I\cap\mathrm{D}(0, 1)|}{\pi}\int_Ardr.
\]
for any bounded Borel sets $A\in\mathbb{R}^+$ and $I\subset\mathbb{C}$. 
\end{theorem}

We omit the proof of Theorem \ref{thm2.1}, which is much simpler compared to Theorem \ref{thm1}. We observe that the typical length of gaps between the eigenvalues from the Ginibre ensemble is $\mathcal{O}(n^{-3/4})$, while for i.i.d. case it is $\mathcal{O}(n^{-1})$. The larger gaps in the Ginibre ensemble demonstrates the repulsive force. Even after normalization, the intensity corresponding to Ginibre ensemble has the $r^3dr$ term, exhibiting heavier tails compared to $rdr$ in the i.i.d. case.

Now let's return to the main theme of the paper. Let $t_1^{(n)}Ê\leq \ldots \leq t_k^{(n)}$ be the $k$ smallest elements in the set $\{|\lambda_i - \lambda_j|: i\neq j\}$. Moreover let $\tau_\ell^{(n)} = (\pi/4)^{1/4}t^{(n)}_\ell$, $\ell = 1, 2\ldots, k$. We have the following corollary which describes the limiting distribution of $\tau_k^{(n)}$. 
\begin{corollary}\label{col1}
For any $0<x_1<y_1<\ldots<x_k<y_k$, as $n\to\infty$
\[
\mathbb{P}(x_\ell<n^{3/4}\tau_\ell^{(n)}<y_\ell, 1\leq\ell\leq k) \longrightarrow \biggl(e^{-x_k^4}-e^{-y_k^4}\biggl)\prod_{\ell=1}^{k-1}(y_\ell^4-x_\ell^4).
\]
In particular, $n^{3/4}\tau^{(n)}_k$ converges in distribution to $\tau_k$ with density
\[
\mathbb{P}(\tau_k\in dx) \propto x^{4k-1}e^{-x^4}dx.
\]
\end{corollary}

For a special case $k =1$, it describes the distribution of the minimum gap between eigenvalues of the Ginibre ensemble. The density is proportional to $x^3e^{-x^4}$.

\subsection{The Complex Wishart Ensemble}
As the second part of the paper, we consider the complex Wishart ensemble. Let the matrix $X_{mn} = (x_{ij})\in\mathbb{C}^{m\times n}$ be a rectangular random matrix of Ginibre ensemble. That is, $x_{ij} = (u_{ij} + \sqrt{-1}v_{ij})/\sqrt{2}$ where $u_{ij}, v_{ij}$'s are  i.i.d. standard normals. A complex Wishart ensemble, denoted as $W_2(m, n)$, can be obtained as $X_{mn}^*X_{mn}$. Here $X^*$ denotes the conjugate transpose of $X$. Since $X_{mn}^*X_{mn}$ and $X_{mn}X_{mn}^*$ share the same non-zero eigenvalues, we can assume without losing generality that $m\geq n$.

Define $A_{mn} = X_{mn}^*X_{mn}/m\in\mathbb{C}^{n\times n}$ as the normalization of the Wishart ensemble. Such form can be recognized as the sample covariance matrix. Denote $\lambda_1\leq\ldots\leq\lambda_n$ as the eigenvalues of $A_{mn}$. As we send $n, m\to\infty$ while keeping the ratio fixed $m/n = \beta\in[1, \infty)$, the empirical distribution of these eigenvalues converges in probability to the March\'{e}nko-Pastur Law (see \cite{8}). The density is 
\[
g(x) = \frac{\beta}{2\pi}\cdot\frac{\sqrt{((1+{\beta^{-1/2}})^2-x)(x-(1-{\beta^{-1/2}})^2)}}{x}.
\]

Similar to the Ginibre case, for any pre-fixed $\epsilon_0>0$, we define the random point process in $\mathbb{R}^+\times\mathbb{R}$ by
\[
\chi^{(n)} = \sum_{i=1}^{n-1}\delta_{(n^{4/3}|\lambda_{i+1}-\lambda_i|, \lambda_i)}\mathbbm{1}_{\{(1-{\beta^{-1/2}})^2+\epsilon_0<\lambda_i<(1+{\beta^{-1/2}})^2-\epsilon_0\}}
\]
Note that here we do not need to define a total ordering because there is already a natural one on $\mathbb{R}$. 

Our second result is stated in the following. 
\begin{theorem}\label{thm1.3}
For any fixed $\epsilon_0>0$, the point process $\chi^{(n)}$ converges to $\chi$ weakly as $m\geq n\to\infty, m/n\to\beta\in[1, \infty)$, in which $\chi$ is a Poisson process in $\mathbb{R}^+\times\mathbb{R}$ with intensity
\[
\mathbb{E}\chi(A\times I) = \frac{\pi^2}{3}\int_Au^2du \int_I\Biggl[\frac{\beta}{2\pi}\frac{\sqrt{((1+{\beta^{-1/2}})^2-x)(x-(1-{\beta^{-1/2}})^2)}}{x}\Biggl]^4dx
\]
for any bounded Borel sets $A\subset\mathbb{R}_+$ and $I\in((1-{\beta^{-1/2}})^2+\epsilon_0, (1+{\beta^{-1/2}})^2-\epsilon_0)$.
\end{theorem}
\begin{remark}
As in \cite{1}, we introduce the $\epsilon_0$ here because of the technical restriction of the Plancherel-Rotach asymptotics of the generalized Laguerre polynomial. In Theorem \ref{thm1}, we successfully removed the $\epsilon_0$ by using more accurate bounds. We believe the $\epsilon_0$ can also be removed here but it need much effort to do so. 
\end{remark}

We can have a similar version of Corollary \ref{col1}. Fix $\epsilon_0>0.$ Let $t_1^{(n)}Ê\leq \ldots \leq t_k^{(n)}$ be the $k$ smallest eigenvalue gaps of the form $\lambda_{i+1}-\lambda_i$ such that $\lambda_i\in((1-{\beta^{-1/2}})^2+\epsilon_0, (1+{\beta^{-1/2}})^2-\epsilon_0)$. Moreover let 
\[
\tau_\ell^{(n)} = \biggl\{\frac{\pi^2}{9}\int_I\Biggl[\frac{\beta\sqrt{((1+{\beta^{-1/2}})^2-x)(x-(1-{\beta^{-1/2}})^2)}}{2\pi x}\Biggl]^4dx\biggl\}^{1/3}t^{(n)}_\ell, \ell = 1, 2\ldots, k. 
\]
We have the following corollary which describes the limiting distribution of $\tau_k^{(n)}$. 
\begin{corollary}\label{col2}
For any $0<x_1<y_1<\ldots<x_k<y_k$, then as $n\to\infty$
\[
\mathbb{P}(x_\ell<n^{4/3}\tau_\ell^{(n)}<y_\ell, 1\leq\ell\leq k) \longrightarrow \biggl(e^{-x_k^3}-e^{-y_k^3}\biggl)\prod_{\ell=1}^{k-1}(y_\ell^3-x_\ell^3).
\]
In particular, the $k$-th smallest normalized gap converges in distribution to $\tau_k$ with density
\[
\mathbb{P}(\tau_k\in dx) \propto x^{3k-1}e^{-x^3}dx.
\]
\end{corollary}

\subsection{The Unitary Universal Ensemble}
In this subsection we define the unitary universal ensemble (UUE), which is a generalization of the GUE. We denote $V(x)$ as a potential function satisfying the following two conditions: $V(x)$ is real analytical on $\mathbb{R}$, and has sufficient growth at $x=\pm\infty$
\[
\lim_{|x|\to\infty}\frac{V(x)}{\log(x^2+1)} = +\infty.
\]
In particular polynomials of the form $V(x) = \sum_{j=0}^k a_jx^{2j}$ satisfy this property. We say a matrix $X_n\in\mathbb{C}^n$ has the universal unitary ensemble, if $X_n$ is hermitian  with the joint probability distribution function
\[
p(X_n) \propto e^{-n\sum_{j=1}^nV(\lambda_j)}
\]
where $\{\lambda_j\}_{j=1}^n$ are the eigenvalues of $X_n$. Using Weyl integration, we know that the joint probability density function for the eigenvalues is 
\[
p(\lambda_1, \ldots, \lambda_n) \propto \prod_{i<j}|\lambda_i-\lambda_j|^2e^{-n\sum_{j=1}^nV(\lambda_j)}.
\]
This ensemble is invariant under unitary transformations (that is, $p(X_n) = p(U_n^*X_nU_n)$ for any fixed unitary matrices $U_n$). That's the reason why we call it ``universal unitary''. Note that if we take $V(x) = x^2$, then it is reduced to the Gaussian unitary ensemble case. But the Wishart ensemble does \emph{not} belong to UUE because the potential function is only defined on $\mathbb{R}_+$

For any potential $V(x)$, there's a related \emph{equilibrium measure}. We briefly define it here for integrity, but for further reference please consult \cite{15}.  Consider the minimization problem 
\[
E_V = \inf_\mu\biggl\{\iint\log|s-t|^{-1}d\mu(s)d\mu(t) + \int V(t)d\mu(t)\biggl|\int_{\mathbb{R}}d\mu = 1\biggl\}
\]
The infimum is achieved uniquely at some measure $\mu = \mu_V$, which is defined to be the equilibrium measure. It is proved that the equilibrium measure has a compact support and has form
\[
d\mu_V(x) = \Psi(x)dx.
\]
where $\Psi(x)$ is strictly positive in the interior of the support.

Now we can establish the following result. 
\begin{theorem}\label{thm1.4}
Suppose the equilibrium measure $\Psi(x)dx$ is supported in a single interval $[a, b]$. For any fixed $\epsilon_0>0$, define
\[
\chi^{(n)} = \sum_{j=1}^n\delta_{(n^{-4/3}|\lambda_{i+1}-\lambda_i|, \lambda_i)}\mathbbm{1}_{\lambda_i\in(a+\epsilon_0, b-\epsilon_0)}.
\] 
Then the point process $\chi^{(n)}$ converges to $\chi$ weakly as $n\to\infty$, in which $\chi$ is a Poisson process in $\mathbb{R}^+\times\mathbb{R}$ with intensity
\[
\mathbb{E}\chi(A\times I) = \frac{\pi^2}{3}\int_Au^2du \int_I\Psi(x)^4dx
\]
for any bounded Borel sets $A\subset\mathbb{R}_+$ and $I\in(a+\epsilon_0, b-\epsilon_0)$.
\end{theorem}

Once again we can derive the $k$ smallest eigenvalue gaps from Theorem \ref{thm1.4}. Fix $\epsilon_0>0.$ Let $t_1^{(n)}Ê\leq \ldots \leq t_k^{(n)}$ be the $k$ smallest eigenvalue gaps of the form $\lambda_{i+1}-\lambda_i$ such that $\lambda_i\in(a+\epsilon_0, b-\epsilon_0)$. Moreover let 
\[
\tau_\ell^{(n)} = \biggl\{\frac{\pi^2}{9}\int_I\Psi(x)^4dx\biggl\}^{1/3}t^{(n)}_\ell, \ell = 1, 2\ldots, k. 
\]
We have the following corollary which describes the limiting distribution of $\tau_k^{(n)}$. 
\begin{corollary}\label{col3}
For any $0<x_1<y_1<\ldots<x_k<y_k$, then as $n\to\infty$
\[
\mathbb{P}(x_\ell<n^{4/3}\tau_\ell^{(n)}<y_\ell, 1\leq\ell\leq k) \longrightarrow \biggl(e^{-x_k^3}-e^{-y_k^3}\biggl)\prod_{\ell=1}^{k-1}(y_\ell^3-x_\ell^3).
\]
In particular, the $k$-th smallest normalized gap converges in distribution to $\tau_k$ with density
\[
\mathbb{P}(\tau_k\in dx) \propto x^{3k-1}e^{-x^3}dx.
\]
\end{corollary}

\section{Outline of the Proof}
Now we denote $\chi^{(n)}$ as defined in (\ref{1.1}). As we will observe from Proposition \ref{prop1}, in order to prove $\chi^{(n)}$ converges to a Poisson point process $\chi$, it remains to prove, for any bounded Borel set $A\subset \mathbb{R}_+$ and $I\subset\mathbb{C}$, we have 
\begin{equation}\label{s2.1}
\chi^{(n)}(A\times I)\to\chi(A\times I) \text{ weakly. }
\end{equation}
However, even proving (\ref{s2.1}) is not easy. Following the ideas by Soshnikov in \cite{5}, we consider the $s$-modified point process. That is, we define 
\[
\widetilde{\xi}^{(n)} := \sum_* \delta_{\lambda_i}
\]
where the we only keep the eigenvalues $\lambda_i$ for which there is exactly \emph{one} eigenvalue $\lambda_j$ lying in $\lambda_i + n^{-3/4}B$ (recall $B = \{u\in\mathbb{C}: |u|\in A, u\succeq0\}$). 

Here's the rationale for $\widetilde{\xi}$. Denote $N = \chi^{(n)}(A\times I)$ and $N_1 = \widetilde{\xi}^{(n)}(I)$. Then, $N_2\neq N$ implies that at least three eigenvalues are clustered together, which is intuitively a rare event. Hence we should have $\chi^{(n)}(A\times I) = \widetilde{\xi}^{(n)}(I)$ asymptotically almost surely. This is proved in Lemma \ref{lemma2}. Now in order to prove \ref{s2.1}, we only need to prove $\widetilde{\xi}^{(n)}(I)\to\chi(A\times I)$ weakly. This is a simplification, because rather than analyzing a two dimensional process ${\chi}^{(n)}$, we can turn to work with a one dimensional process $\widetilde{\xi}^{(n)}$ instead.

The process $\widetilde{\xi}$ may seem complicated, but its $k$ point correlation function $\widetilde{\rho}_k(\lambda_1, \ldots, \lambda_k)$ can be expressed explicitly using the inclusion-exclusion argument, see (\ref{eqn5}).
Again, by Proposition \ref{prop1} the problem is reduced to prove that the corresponding factorial moment matches each other, i.e., for any integer $k>0$, as $n\to\infty$
\begin{equation}\label{s2.2}
\mathbb{E}\biggl[\frac{\xi^{(n)}(I)!}{(\xi^{(n)}(I)-k)!}\biggl] = \int_{I^k}\widetilde{\rho}_k(\lambda_1, \ldots, \lambda_k)d\lambda_1\ldots d\lambda_k \to \int_{I^k}\biggl(\frac{1}{\pi^2}|u|^2du\biggl)^k\cdot\prod_{i=1}^k\mathbbm{1}_{|\lambda_i|<1}.
\end{equation}
Indeed, the point-wise convergence of $\widetilde{\rho}_k$ to the integrand of the right hand side of (\ref{s2.2}) is proved in Lemma \ref{lemma3}. Now the only thing to do is to swap the limit and integration. To do this, we prove that $\widetilde{\rho}_k$ can be uniformly bounded in the majority of the region (see Lemma \ref{lemma4}) so that we can use the Lebesgue's dominated convergence theorem. On the other hand, although the $\widetilde{\rho}_k$ will explode in the rest of the minor region, we prove that since the minor region is so small, the integration of $\widetilde{\rho}_k$ on that will have a negligible effect (see Lemma \ref{lemma5}). Combining all these, we finished our proof.

\section{Proofs for the Ginibre Ensemble Case}
In this section we prove Theorem \ref{thm1} and Corollary \ref{col1}. 
We know from \cite{6} that for $X_n$ being the Ginibre ensemble, the eigenvalues of $X_n/\sqrt{n}$ has joint probability distribution
\[
p(\lambda_1, \ldots, \lambda_n) = C_n\exp\biggl(-n\sum_{i=1}^n|\lambda_i|^2\biggl)\prod_{1\leq i<j\leq n}|\lambda_i-\lambda_j|^2
\]
where $C_n$ is a constant. Moreover we can define its $k$-point correlation function as
\[
\rho_k(\lambda_1, \ldots, \lambda_k) := \frac{n!}{(n-k)!}\int_{-\infty}^{\infty}\ldots\int_{-\infty}^{\infty}p(\lambda_1, \ldots, \lambda_n)d\lambda_{k+1}\ldots d\lambda_n.
\]
Intuitively, $\rho_k(\lambda_1, \ldots, \lambda_k)d\lambda_1\ldots d\lambda_k$ is the expectation of the number of $k$ tuples of eigenvalues in $(\lambda_1+d\lambda_1)\cup\ldots\cup(\lambda_k+d\lambda_k)$. Note $\rho_k$ is \emph{not} a probability density function as the total integration is not unity. The following property is quite useful in this paper. If we denote $N$ as the number of eigenvalues in some interval $I$, then for any integer $k\leq n$ we have
\begin{equation}\label{v3.1}
\mathbb{E}\biggl[\frac{N!}{(N-k)!}\biggl] = \int_{I^k}\rho_k(\lambda_1, \ldots, \lambda_k)d\lambda_1, \ldots, \lambda_k.
\end{equation}
For a more detailed explanation, please refer to \cite{6}. 

By using the technique of orthogonal polynomials, we can write $\rho_k$ in a determinantal form. 
\[
\rho_k(\lambda_1, \ldots, \lambda_k) = n^k\pi^{-k}\exp\biggl(-n\sum_{j=1}^k|\lambda_j|^2\biggl)\det(K_n(\lambda_i\lambda_j^*))_{1\leq i, j\leq k}.
\]
Here 
\[
K_n(x) = \sum_{\ell = 0}^{n-1}\frac{n^\ell x^\ell}{\ell!}.
\]

Before we begin proving the result, we first analyze the behavior of $K_n(x)$. This will play an essential role in later proofs. In short, for $x$ being real, $e^{-nx}K(x)\to 1$ if $x\in(0, 1)$ and $e^{-nx}K(x)\to 0$ for $x>1$.
For $x$ being complex, the following lemma gives a precise bound on the decay rate.


\begin{lemma}\label{good}
Let $z = z_n = x + n^{-3/4}c$. Here $x\in\mathbb{R}$ and $c\in\mathbb{C}$ are independent of $n$. Moreover let
\[
K_n(z) = \sum_{\ell=0}^{n-1}\frac{n^\ell z^\ell}{\ell!} := e^{nz}( 1 - R_n(z)),
\]
Then the following inequalities are valid.
\begin{enumerate}
\item
If $|z|\leq 0.02$, then
\[
|R_n(z)| \leq \sqrt{\frac{n}{2\pi}}0.06^n.
\]
\item
If $0.01<|z|\leq1$, then there exists constants $C$ and $N$ (independent of $z$ and $n$) such that for all $n>N$, 
\[
|R_n(z)| \leq C{\sqrt{n}}\Bigl(|z|e^{1-|z|}\Bigl)^n.
\]
\item 
If $|z|\geq 1$, then there exists constants $C$ and $N$ (independent of $z$ and $n$) such that for all $n>N$, 
\[
|1-R_n(z)|\leq C\Bigl(|z|e^{-(|z|-1)}\Bigl)^n.
\]
\end{enumerate}
\end{lemma}

%
%
\begin{proof} 

\textsc{Part 1.}

By applying Taylor's theorem in the complex setting on $f(z) := e^{nz}$ we have
\begin{eqnarray*}
R_n(z) & = & \int_0^z \frac{(z-t)^{n-1}n^ne^{-n(z-t)}}{(n-1)!}dt.
\end{eqnarray*}
Here the integral is along any curve connecting $0$ and $z$. In particular, we can set the curve to be the straight line between the two points.
By Stirling's theorem, $n! = \sqrt{2\pi n}(n/e)^ne^{\beta_n}$ for some $\beta_n \in (12n+1)^{-1}, (12n)^{-1})$. Thus 
\begin{eqnarray*}
R_n(z) & = & \sqrt{\frac{n}{2\pi}}e^{n-\beta_n}\int_0^z(z-t)^{n-1}e^{-n(z-t)}dt \\
& = & \sqrt{\frac{n}{2\pi}}e^{n-\beta_n}\int_0^zt^{n-1}e^{-nt}dt \\
& = &  \sqrt{\frac{n}{2\pi}}z^ne^{n-\beta_n}\int_0^1t^{n-1}e^{-ntz}dt. 
\end{eqnarray*}
Hence
\begin{eqnarray*}
|R_n(z)| &\leq & \sqrt{\frac{n}{2\pi}}|z|^ne^n\int_0^1t^{n-1}e^{-ntu}dt \\
\end{eqnarray*}
Here we denote $u = \Re(z)$ and $v = \Im(z)$. If $|z|\leq 0.02,$ then $e^{-ntu}\leq e^{n|z|}\leq e^{0.02n}$. Hence
\[
|R_n(z)| \leq \sqrt{\frac{n}{2\pi}}e^{1.02n}|z|^n \leq \sqrt{\frac{n}{2\pi}}\Bigl(0.02e^{1.02}\Bigl)^n\leq \sqrt{\frac{n}{2\pi}}0.06^n.
\]

\textsc{Part 2.}
If $0.01<|z|\leq 1$, then for sufficiently large $n$, $u = \Re(z)>0.005>0$. Thus from part (1) of the lemma
\begin{eqnarray*}
|R_n(z)| &\leq & \sqrt{\frac{n}{2\pi}}|z|^ne^n\int_0^1t^{n-1}e^{-ntu}dt = \sqrt{\frac{n}{2\pi}}|z|^ne^n\int_0^1g(t)dt \\
\end{eqnarray*}
where $g(t) = t^{n-1}e^{-ntu}$. By standard analysis, $g(t)$ achieves its maximum at $t^* = (n-1)/(nu)$. Thus, if $t^*\leq1$, i.e., $n-1\leq nu$,
\begin{eqnarray*}
|R_n(z)| & \leq & \sqrt{\frac{n}{2\pi}}|z|^ne^n g(t^*) = \sqrt{\frac{n}{2\pi}}|z|^n\exp\biggl(1+(n-1)\log\frac{n-1}{nu}\biggl) \\
& \leq &  \sqrt{\frac{n}{2\pi}}e|z|^n \leq \sqrt{\frac{n}{2\pi}}e|z|^ne^{n(1-|z|)}.
\end{eqnarray*}
On the other hand, if $t^*>1$, then 
\begin{eqnarray*}
|R_n(z)| & \leq & \sqrt{\frac{n}{2\pi}}|z|^ne^n g(1) = \sqrt{\frac{n}{2\pi}}|z|^ne^{n(1-u)} \\
& = &  \sqrt{\frac{n}{2\pi}}|z|^ne^{n(1-|z|)}\cdot e^{n(|z|-u)}.
\end{eqnarray*}
Let $z = |z|e^{i\theta}$. Then $\exp\Big(n(|z|-u)\Big) = \exp\Big(n|z|(1-\cos\theta)\Big)$. Since\\ $\theta = \arcsin(\Im(z)/|z|) = \mathcal{O}(n^{-3/4}/|z|)$, we obtain $n|z|(1-\cos\theta) \sim n|z|\theta^2/2 \sim \mathcal{O}(n^{-1/2}/|z|)$.
Hence $\exp\Big(n(|z|-u)\Big)\to 1$ as $n\to\infty$.

Thus we have complete proving the existence of these constants $N$ and $C$.


\textsc{Part 3.}
If $|z|\geq1$, the quantity $|1-R_n(z)|$ has the following upper bound (recall $u = \Re(z)$).
\[
|1-R_n(z)|  =  \Biggl|e^{-nz}\sum_{\ell=0}^{n-1}\frac{n^\ell z^\ell}{\ell!}\Biggl| \leq e^{-nu}\sum_{\ell=0}^{n-1}\frac{n^\ell |z|^\ell}{\ell!}
\]
Let $A_\ell = n^{\ell}|z|^{\ell}/\ell!$ be the $\ell$-th term in the above summation. Then
\[
A_{\ell} \leq \frac{n-1}{n}\cdot\frac{n-2}{n}\cdot\cdots\cdot\frac{\ell+1}{n}\cdot A_{n-1} := c_\ell A_{n-1}
\]
Thus $|1-R_n(x)|\leq e^{-nu}A_{n-1}\cdot\sum_{\ell=0}^{n-1}c_\ell$. Our next target is to prove $\sum_{\ell = 0}^{n-1}c_\ell /\sqrt{n} $ is bounded above. If this can be done, then by Stirling's formula
\begin{eqnarray*}
|1-R_n(z)| &\leq & \mathcal{O}(1)\cdot e^{-nu}\frac{n^n|z|^n}{n!}\cdot\sqrt{n} \leq \mathcal{O}(1)\cdot \Bigl(|z|e^{1-u}\Bigl)^n \\
& =& \mathcal{O}(1)\cdot \Bigl(|z|e^{1-|z|}\Bigl)^n\cdot e^{n(|z|-u)}.
\end{eqnarray*}
By part (b) of the lemma, $\exp\Big(n(|z|-u)\Big)\to 1$ as $n\to\infty$. This proves the existence of the constants $N$ and $C$.

The only thing left is to prove $\sum_{\ell=0}^{n-1}c_\ell/\sqrt{n}$ is bounded above. Indeed, for $\ell\geq n-\sqrt{n}, c_\ell\leq 1$. For $\ell < n- \sqrt{n}$, $c_\ell \leq (1-1/\sqrt{n})^{n-\sqrt{n}-\ell}$. Thus we get
\begin{eqnarray*}
\sum_{\ell = 0}^{n-1}c_\ell & = & \sum_{\ell \geq n-\sqrt{n}}c_\ell + \sum_{\ell < n-\sqrt{n}}c_\ell 
 \leq  \sqrt{n} + \sum_{\ell = 0}^{\infty}\biggl(1-\frac{1}{\sqrt{n}}\biggl)^\ell \leq 2\sqrt{n}.
\end{eqnarray*}
\end{proof}


\begin{remark}\label{remark1}
At the end of the proof of part 2 of Lemma (\ref{good}), we actually proved a quite useful result. That is, $e^{n(\Re(\xi) - |\xi|)} \to 1$ uniformly if $\xi$ is bounded away from zero and $\Im(\xi) = \mathcal{O}(n^{-3/4})$. This will be used for several times later on. 
\end{remark}

\begin{remark}\label{remark2}
Since the one-point correlation function is 
\[
\rho_1(\lambda_1) = n\pi^{-1}\exp(-n|\lambda_1|^2)K_n(|\lambda_1|^2),
\] 
From Lemma (\ref{good}) we know that $\rho_1(\lambda_1)/n$ converges to $\pi^{-1}$ if $|\lambda_1|<1$ and zero if $|\lambda_1|>1$. This is in agreement with the famous circular law. 
\end{remark}

In order to prove the convergence of the random point process $\chi^{(n)}$, 
we give the following proposition. This is a very slight modification of the Proposition 2.1 in \cite{1}. The proofs can also be found there.
\begin{proposition}\label{prop1}
Let $\chi^{(n)} = \sum_{i=1}^{k_n}\delta_{X_{i, n}}$ be a sequence of point processes on $\mathbb{R}^+\times\mathbb{C}$, and $\chi$ a Poisson point process on $\mathbb{R}^+\times\mathbb{C}$ with intensity $\mu$ having no atoms and $\sigma$-finite. Assume that for any bounded intervals $A$ and $I$ and all possible $k\geq1$
\begin{equation}\label{eqn1}
\lim_{n\to\infty}\mathbb{E}\biggl(\frac{\chi^{(n)}(A\times I)!}{(\chi^{(n)}(A\times I)-k)!}\biggl) = \mu(A\times I)^k.
\end{equation}
Then the sequence of point processes $\chi^{(n)}$ converges in distribution to $\chi$.
\end{proposition}

Moreover the following lemma is frequently used. The proofs can be found in \cite{9}.
\begin{lemma}\label{lemma1}
Let $M$ be an $n\times n$ Hermitian positive definite matrix. For any $\omega\subseteq\{1, 2, \ldots, n\}$, let $M_\omega$ (resp. $M_{\overline\omega}$) be the submatrix of $M$ using rows and columns numbered in $\omega$ (resp. $\{1, \ldots, n\}\backslash\omega$). Then
\[
\det(M)\leq\det(M_\omega)\det(M_{\overline{\omega}}).
\]
\end{lemma}


Now we begin to prove Theorem \ref{thm1}, that the point process $\chi^{(n)} = \sum_{i}\delta_{n^{-3/4}|\lambda_{i^*}-\lambda_i|, \lambda_i}$ converges to a Poisson process. By Proposition \ref{prop1}, we just need to verify (\ref{eqn1}). To do this, it remains to prove that for any fixed bounded sets $A\subset\mathbb{R}^+$ and $I\subset\mathbb{C}$,
\begin{equation}\label{eqn2}
\chi^{(n)}(A\times I)\overset{\mathcal{D}}{\longrightarrow}\mathrm{Poisson}(\lambda)
\end{equation}
for $\lambda = \frac{1}{\pi^2}(\int_B|u|^2du)(\int_{I\cap\mathrm{D}(0, 1)}dv)$. 

Let $A_n = n^{-3/4}A = \{n^{-3/4}a: a\in A\}$. $B_n = n^{-3/4}B$ (recall $B = \{u\in\mathbb{C}: |u|\in A, u\succeq0\}$).
Consider the point process 
\[
\xi^{(n)} := \sum_{i=1}^{n-1}\delta_{\lambda_i}
\]
and its thinning $\widetilde{\xi}^{(n)}$ obtained from $\xi^{(n)}$ by only keeping the eigenvalues $\lambda_k$ for which $\xi^{(n)}(\lambda_k+B_n) = 1$. This technique is introduced in \cite{5} by Soshinikov as the \emph{$s$-modified} process. Intuitively, for any bounded sets $A\subset\mathbb{R}^+$ and $I\subset\mathbb{R}^2$, we have
 \begin{eqnarray*}
 \chi^{(n)}(A\times I) &=& \#\{i: \lambda_i\in I, |\lambda_{i^*} - \lambda_i|\in A_n\} \\
\widetilde{\xi}^{(n)}(I) &=& \#\{i: \lambda_i\in I, \text{ there exists only one }\lambda_j \in \lambda_i + B_n\}.
 \end{eqnarray*}
 Then $\{\chi^{(n)}(A\times I)\neq\widetilde{\xi}^{(n)}(I)\}$ would imply that there exist at least two eigenvalues  clustering around the third eigenvalue $\lambda_i$, which is a rare event. Hence we expect $\chi^{(n)}(A\times I) = \widetilde{\xi}^{(n)}(I)$ asymptotically almost surely. This gives us an advantage: instead of analyzing the $2$-D point process $\chi^{(n)}$, we can turn to work with the $1$-D process $\widetilde{\xi}^{(n)}$, which is much simpler.

To formalize this, the following lemma says that $\chi^{(n)}(A\times I)$ and $\widetilde{\xi}^{(n)}(I)$ are asymptotically the same.


\begin{lemma}\label{lemma2}
For any bounded sets $A\subset\mathbb{R}^+$ and $I\subset\mathbb{R}^2$,
\[
\chi^{(n)}(A\times I)-\widetilde{\xi}^{(n)}(I)\overset{\mathcal{D}}{\longrightarrow}0.
\]
\end{lemma}

%
%

\begin{proof} Let $c$ be such that $A\subset (0, c)$ and define $c_n = cn^{-3/4}$. 

First we prove that if 
\[
\mathbbm{1}_{\{\lambda_{i^*}-\lambda_i\in B_n\}}\neq\mathbbm{1}_{\{\xi^{(n)}(\lambda_i+B_n)=1\}},
\]
then $\xi^{(n)}(\lambda_i+\mathrm{D}^+(0, c_n))\geq2$, that is, at least 3 eigenvalues are clustered together. Here $\mathrm{D}^+(0, c_n) = \{z: |z|\leq c_n, z\succeq0\}$ is the half disk in the complex plane.

Indeed, if $\lambda_{i^*} - \lambda_i\in B_n$(recall the $\lambda_{i^*}$ is the closest eigenvalue to $\lambda_i$) and $\xi^{(n)}(\lambda_i+B_n)\neq1$, then $\xi^{(n)}(\lambda_i+B_n)>1$. Thus $\xi^{(n)}(\lambda_i+\mathrm{D}^+(0, c_n))\geq\xi^{(n)}(\lambda_i+B_n)\geq2$. Conversely, if $\xi^{(n)}(\lambda_i+B_n)=1$ then there exist $j$ such that $\lambda_j-\lambda_i\in B_n\subset\mathrm{D}^+(0, c_n)$. Thus $\lambda_{i^*}\in\lambda_i + \mathrm{D}^+(0, c_n)$. If $\lambda_{i^*}\notin B_n$ then $\lambda_{i^*}\neq\lambda_j$. This implies that $\xi^{(n)}(\lambda_i+\mathrm{D}^+(0, c_n))\geq2$. In summary, no matter what case happens, we always have $\xi^{(n)}(\lambda_i+\mathrm{D}^+(0, c_n))\geq2$. Thus the statement is correct. 

From the statement we have
\[
|\chi^{(n)}(A\times I)-\widetilde{\xi}^{(n)}(I)|\leq\sum_{i=1}^{n-1}\mathbbm{1}_{\{\xi^{(n)}(\mathrm{D}^+(\lambda_i, c_n))\geq2\}}\leq\Xi^{(3)}(\mathcal{B}).
\]
where 
\[
\Xi^{(k)} = \sum_{\lambda_{i_1}, \ldots, \lambda_{i_k} \text{ are all distinct }}\delta_{(\lambda_{i_1}, \ldots, \lambda_{i_k})}
\]
and
\[
\mathcal{B} = \{(\lambda, x_1, x_2): \lambda\in I \text{ and } (x_1, x_2)\in\mathrm{D}^+(0, c_n)^2\}.
\]
Since $L^1$ convergence implies weak convergence, we just need to prove $\mathbb{E}(\Xi^{(3)}(\mathcal{B}))\to 0$ as $n\to \infty$. We have 
\begin{equation}\label{v3.3}
\mathbb{E}(\Xi^{(3)}(\mathcal{B})) = \int_Id\lambda\int_{\mathrm{D}^+(\lambda, c_n)^2}\rho_3(\lambda, x_1, x_2)dx_1dx_2
\end{equation}
where $\rho_3$ is the $3$-point correlation function. That is,
\begin{multline}\label{v3.2}
\rho_3(\lambda, x_1, x_2) = n^3\pi^{-3}\exp(-n(|\lambda|^2+|x_1|^2+|x_2|^2))\times \\\det
\left(
\begin{array}{ccc}
K_n(|\lambda|^2) & K_n(\lambda x_1^*) & K_n(\lambda x_2^*)  \\
K_n(x_1\lambda^*) & K_n(|x_1|^2) & K_n(x_1x_2^*) \\
K_n(x_2\lambda^*) & K_n(x_2x_1^*) & K_n(|x_2|^2)
\end{array}
\right).
\end{multline}

In order to prove that (\ref{v3.3}) approaches to zero, we need to decompose it into three parts because of the different asymptotics of $K_n(z)$ in different regions. More precisely, we define 
\begin{eqnarray*}
I_1 & = & \{\lambda\in I : |\lambda|\leq 1-n^{-0.02}\}, \\
I_2 & = & \{\lambda\in I : |\lambda|\geq 1+n^{-0.02}\}, \\
I_3 & = & \{\lambda\in I : 1-n^{-0.02}<|\lambda|< 1+n^{-0.02}\}
\end{eqnarray*}
as the inner part, outer part and the middle part of $I$. 
Below we shall prove that the integral is small on all the three regions. Our strategy is as follows.
\begin{enumerate}
\item
For the inner part, $K_n(z)\sim e^{nz}$. We expect that the terms in the determinant in (\ref{v3.2}) would cancel out, leaving us a small integrand.
\item
For the outer part, $K_n(z) \sim 0$. Hence each term in the determinant is very small. 
\item
For the boundary part, we can only obtain a poor upper bound for $\rho_3(\lambda_1, \lambda_2, \lambda_3)$. However this time the integral region of the boundary part is small, which yields a small result. 
\end{enumerate}

Below is our rigorous proof.

%
%
\textsc{ (1) Inner Part.}
Let's first start with $I_1$.
Replace $K_n(z)$ with $e^{nz}(1-R_n(z))$ in (\ref{v3.2}), we get
\begin{multline}\label{eq1}
 \rho_3(\lambda, x_1, x_2) = Q_n(\lambda, x_1, x_2)
 + \\  n^3\pi^{-3}\exp(-n(|\lambda|^2+|x_1|^2+|x_2|^2))
\left|
\begin{array}{ccc}
\exp(n|\lambda|^2) & \exp(n\lambda x_1^*) & \exp(n\lambda x_2^*) \\
\exp(nx_1\lambda^*) & \exp(n|x_1|^2) & \exp(nx_1x_2^*) \\
\exp(nx_2\lambda^*) & \exp(nx_2x_1^*) & \exp(n|x_2|^2) 
\end{array}
\right|.
\end{multline}
Here $Q_n(\lambda, x_1, x_2)$ is the residual, being the summation of several terms. One typical term is
\begin{equation}\label{eq11}
n^3\pi^{-3}\exp\Bigl(-n(|\lambda|^2 + |x_1|^2 + |x_2|^2)\Bigl)\cdot \exp(n\lambda x_1^* + nx_1x_2^* + nx_2\lambda^*)R_n(\lambda x_1^*).
\end{equation}
Taking the absolute values, we get $|(\ref{eq11})|\leq n^3\pi^{-3}|R_n(\lambda x_1^*)|$. Before providing a bound on the main part of (\ref{eq1}), we first prove that the residual term (\ref{eq11}) is small.

First we consider the case where $|\lambda| < 0.12$ is very small. In this case for $n$ sufficiently large $|\lambda x_1^*| < 0.02$. Thus by part (1) of Lemma \ref{good}, $R_n(\lambda x_1^*)$ uniformly decays exponentially, hence the impact of (\ref{eq11}) is negligible as $n$ tends to infinity.

Next we consider the case $|\lambda| \geq0.12$. Since $|\lambda|<1-n^{-0.02}$ and $|x_1-\lambda|<n^{-3/4}$, we obtain that for $n$ sufficiently large, $0.01 < |x_1\lambda^*|<1-n^{-0.01}$. Thus by part (2) of Lemma \ref{good}, the upper bound for the absolute value is
\begin{eqnarray*}
n^3\pi^{-3}|R_n(\lambda x_1^*)| & \leq & \mathcal{O}(1)\cdot n^{3}\Bigl(|x_1\lambda^*|e^{1-|x_1\lambda^*|}\Bigl)^n \\
& = & \mathcal{O}(1)\cdot\exp\Bigl(n\log|x_1\lambda^*|+n-n|x_1\lambda^*|+3\log n\Bigl) \\
& \leq & \mathcal{O}(1)\cdot\exp\Bigl(n\log(1-n^{-0.01})+n^{0.99}+3\log n\Bigl) \\
& = & \mathcal{O}(1)\cdot\exp\Bigl(-\frac{1}{2}n^{0.98} + o(n^{0.98})\Bigl) \to 0.
\end{eqnarray*}
The other terms converge to zero uniformly as well. So we get that $Q_n(\lambda, x_1, x_2)$ converge to zero uniformly --- their impact is negligible.

Next we turn to analyze the main part of (\ref{eq1}). Expanding the determinant, we have
\begin{eqnarray*}
\rho_3(\lambda, x_1, x_2)& =& \pi^{-3}n^3(3-e^{-n|x_2-\lambda|^2}-e^{-n|x_1-\lambda|^2}-e^{-n|x_2-x_1|^2}) \\ 
& &+\pi^{-3}n^3(e^{n(x_1\lambda^*+x_2x_1^*+\lambda x_2^*-|\lambda|^2-|x_1|^2-|x_2^2|)}-1)\\
& &+\pi^{-3}n^3(e^{n(x_1x_2^*+x_2\lambda^*+\lambda x_1^*-|\lambda|^2-|x_1|^2-|x_2^2|)}-1) + o(1).
\end{eqnarray*}
In the expression above, the first term is $\mathcal{O}(n^3)\times \mathcal{O}(n^{-1/2}) = \mathcal{O}(n^{5/2})$. To analyze the second term, let $x_1 = \lambda + n^{-3/4}u_1, x_2 = \lambda + n^{-3/4}u_2.$ Then, after some calculation, the second term is
\[
\pi^{-3}n^3\biggl[\exp\Bigl(-n^{-1/2}(|u_1|^2 + |u_2|^2 - u_2u_1^*)\Bigl)-1\biggl] = \mathcal{O}(n^{5/2}).
\]
The same is true for the third term. Thus $\rho_3(\lambda, x_1, x_2) = \mathcal{O}(n^{5/2})$. But the integration domain $I_1\times\mathrm{D}^+(\lambda_i, c_n)^2$ is of order $\mathcal{O}(n^{-3})$. Hence we successfully proved that uniformly
\[
\lim_{n\to\infty}\int_{I_1}d\lambda\int_{\mathrm{D}^+(\lambda, c_n)^2}\rho_3(\lambda, x_1, x_2)dx_1dx_2 = 0.
\]

%
%
\textsc{(2) Outer Part.}
Next we consider the integration on $I_2$. The integration domain is $\mathcal{O}(n^{-3})$, but we prove that the integrand $\rho_3(\lambda, x_1, x_2)$ is $o(1)$. Indeed, expanding the determinant of $\rho_3(\lambda, x_1, x_2)$ in (\ref{v3.2}), we get a summation of several terms. One typical term is
\begin{equation}\label{eq2}
n^3\pi^{-3}\exp(-n(|\lambda|^2+|x_1|^2+|x_2|^2))\cdot K_n(\lambda x_1^*)K_n(x_1x_2^*)K_n(x_2\lambda^*).
\end{equation}
For $|\lambda|>1+n^{-0.02}$ and $|x_i-\lambda|<n^{-3/4}$, we have that for $n$ sufficiently large, $|x_1x_2| > 1+n^{-0.01}$ and $|x_i\lambda^*|>1+n^{-0.01}$ for $i = 1, 2.$ Hence by part (3) of Lemma \ref{good} the magnitude of (\ref{eq2}) has upper bound
\begin{eqnarray*}
|(\ref{eq2})| & \leq & \mathcal{O}(n^3)\Bigl(|\lambda x_1^*|e^{-(|\lambda x_1^*|-1)}\Bigl)^n\Bigl(|x_1x_2^*|e^{-(|x_1x_2^*|-1)}\Bigl)^n\Bigl(|x_2\lambda^*|e^{-(|x_2\lambda^*|-1)}\Bigl)^n\\
&\leq & \mathcal{O}(1)\cdot \exp(3n\log(1+n^{-0.01}) - 3n^{0.99} + 3\log n) \\
& = & \exp(-3n^{0.98}/2 + o(n^{0.98}))\to 0.
\end{eqnarray*}
Similarly, the other terms in the expansion of the determinant  also uniformly converge to zero. Hence $\rho(\lambda, x_1, x_2) = o(1)$ uniformly. Thus we conclude that the integral on $I_2$ also converge to zero.

%
%
\textsc{(3) Middle Part.}
Finally, we prove the integral on $I_3$ is also of $o(1)$. Again we expand the determinant and prove that each term is small. Since for any $z\in\mathbb{C}$, 
\begin{equation}\label{v3.4}
|K_n(z)|\leq \sum_{\ell = 0}^{n-1}\frac{n^\ell|z|^\ell}{\ell!} \leq \sum_{\ell = 0}^{\infty}\frac{n^\ell|z|^\ell}{\ell!}  = e^{n|z|},
\end{equation}
Expanding the determint in (\ref{v3.2}) again and using the estimate (\ref{v3.4}), we have that the magnitude of one typical term in the expansion has upper bound
\begin{eqnarray*}
&& n^3\pi^{-3}\exp(-n(|\lambda|^2 + |x_1|^2 + |x_2|^2))|K_n(\lambda x_1^*)K_n(x_1x_2^*)K_n(x_2\lambda^*)| \\
&\leq & n^3\pi^{-3}\exp\biggl\{-\frac{n(|\lambda|-|x_1|)^2}{2}-\frac{n(|\lambda|-|x_2|)^2}{2}-\frac{n(|x_1|-|x_2|)^2}{2}\biggl\} = \mathcal{O}(n^3).
\end{eqnarray*}
But the integration domain is of order $\mathcal{O}(n^{-3})\times \mathcal{O}(n^{-0.02}) = \mathcal{O}(n^{-3.02})$. Hence the total integration over $I_3$ decay to zero at the speed of $\mathcal{O}(n^{-0.02})$. 

From all the above, we proved that the integral of $I = I_1\cup I_2\cup I_3$ converges to zero. Hence $\mathbb{E}(\Xi^{(3)}(\mathcal{B}))\to 0.$ 
\end{proof}


Note that previously our ultimate goal is to prove $\chi^{(n)}(A\times I)\to\mathrm{Poisson}(\lambda)$ for $\lambda = \frac{|I\cap\mathrm{D}(0, 1)|}{\pi}\int_Ar^3dr$. Now Lemma \ref{lemma2} tells us $\chi^{(n)}(A\times I)$ and $\widetilde{\xi}^{(n)}$ has asymptotically the same distribution, it remains to prove $\widetilde{\xi}^{(n)}(I)\to\mathrm{Poisson}(\lambda)$. Or, we just need to prove that 
\begin{equation}\label{eqn3}
\mathbb{E}\biggl[\frac{\widetilde{\xi}^{(n)}(I)!}{(\widetilde{\xi}^{(n)}(I)-k)!}\biggl] \overset{n\to\infty}{\longrightarrow} \biggl(\frac{1}{\pi^2}\int_B|u|^2du\biggl)^k\biggl(\int_{I\cap\mathrm{D}(0, 1)}dv\biggl)^k, \forall k\geq1.
\end{equation}
This is because if (\ref{eqn3}) holds, then all the moments of $\widetilde{\xi}^{(n)}(I)$ converges to the moment of $\mathrm{Poisson}(\lambda)$. This implies $\widetilde{\xi}^{(n)}(I)\to\chi(A\times I)$ weakly, or $\chi^{(n)}(A\times I)\to\chi(A\times I)$ weakly. 

Denote $\widetilde{\rho}_k(\lambda_1, \ldots, \lambda_k)$ as the $k$-point correlation function of $\widetilde{\xi}^{(n)}$. Recall the point correlation function has the property (\ref{v3.1}), the left hand side of (\ref{eqn3}) can be expressed as
\begin{equation}\label{eqn4}
\mathbb{E}\biggl[\frac{\widetilde{\xi}^{(n)}(I)!}{(\widetilde{\xi}^{(n)}(I)-k)!}\biggl] = \int_{I^k}\widetilde{\rho}_k(\lambda_1, \ldots, \lambda_k)d\lambda_1\ldots d\lambda_k.
\end{equation}
Thus now our ultimate goal is to prove 
\begin{equation}\label{v3.5}
\int_{I^k}\widetilde{\rho}_k(\lambda_1, \ldots, \lambda_k)d\lambda_1\ldots d\lambda_k \overset{n\to\infty}{\longrightarrow} \biggl(\frac{1}{\pi^2}\int_B|u|^2du\biggl)^k\biggl(\int_{I\cap\mathrm{D}(0, 1)}dv\biggl)^k, \forall k\geq1.
\end{equation}

To establish (\ref{v3.5}), it remains to prove the following three statements.
\begin{enumerate}
\item
If all the $\lambda_k$'s are distinct and have magnitude not equal to 1, then 
$
\widetilde{\rho}(\lambda_1, \ldots, \lambda_k) \to
(\frac{1}{\pi^2}\int_B|u|^2du)^k\cdot\prod_{i=1}^k\mathbbm{1}_{|\lambda_i|<1}.
$
\item
Define 
\[
\Omega = \{(\lambda_1, \ldots, \lambda_k)\in I^k: (\lambda_i+B_n)\cap(\lambda_j+B_n) = \emptyset, 1\leq i, j\leq k\}.
\] 
Then $\widetilde{\rho}(\lambda_1, \ldots, \lambda_k)$ is uniformly bounded in $\Omega$ (Lemma \ref{lemma4}).
\item
Define $\overline{\Omega} = I^k\backslash\Omega$, then the contribution of the integral of $\widetilde{\rho}_k(\lambda_1, \ldots, \lambda_k)$ in $\overline{\Omega}$ is negligible since the volume of $\overline{\Omega}$ is sufficiently small (Lemma \ref{lemma5}).
\end{enumerate}

As a little discussion, the set $\Omega$ represents the majority of the integration region. Combining the point-wise convergence result in statement 1 and the boundedness result in statement 2, we can perform the dominated convergence theorem in the main region $\Omega$. Furthermore statement 3 ensures that the contribution of the minor region $\overline{\Omega}$ does not play an important role, hence we can obtain the result (\ref{v3.5}).

We prove the three statements one by one in the following.


\begin{lemma}\label{lemma3}
(Point-wise convergence) Let $\lambda_1, \ldots, \lambda_k$ be distinct complex numbers and $|\lambda_i|\neq1$ for $i = 1, \ldots, k$. Then as $n\to\infty$
\[
\widetilde{\rho}_k(\lambda_1, \ldots, \lambda_k) \longrightarrow
\biggl(\displaystyle\frac{1}{\pi^2}\int_B|u|^2du\biggl)^k\cdot\prod_{i=1}^k\mathbbm{1}_{|\lambda_i|<1}.
\]
\end{lemma}

%
%

\begin{proof} To begin with we consider the first case where $|\lambda_i|<1$ for all $i$. Since all the $\lambda_i$'s are distinct, for $n$ large enough, $\lambda_i+B_n$ are disjoint sets for different $i$'s. Moreover there exists some $\epsilon_0>0$ such that for $n$ sufficiently large, $\lambda_i+B_n\subset \mathrm{D}(0, 1-\epsilon_0)$.

By an inclusion-exclusion argument, for sufficiently large $n$ we can represent $\widetilde{\rho}_k$ in terms of $\rho_k$ as follows. For a general result please see \cite{5}.
\begin{multline}\label{eqn5}
\widetilde{\rho}_k(\lambda_1, \ldots, \lambda_k) = \sum_{m=0}^\infty\frac{(-1)^m}{m!}\int_{\lambda_1+B_n}dx_1\ldots\int_{\lambda_k+B_n}dx_k \\
\int_{((\lambda_1+B_n)\sqcup\ldots\sqcup(\lambda_k+B_n))^m}\rho_{2k+m}(\lambda_1, x_1, \ldots, \lambda_k, x_k, y_1, \ldots, y_m)dy_1\ldots dy_m.
\end{multline}

Since $\rho_{2k+m}=0$ for $2k+m>n$, (\ref{eqn5}) is a finite sum thus there is no convergence issue. First of all we consider the $m=0$ case. Now the expression for $\rho_{2k}$ is
\begin{multline}
\rho_{2k}(\lambda_1, x_1, \ldots, \lambda_k, x_k) \\ = n^{2k}\pi^{-2k}\exp\biggl(-n\sum_{j=1}^k(|\lambda_j|^2+|x_j|^2)\biggl)\det_{1\leq i, j\leq k}
\left(
\begin{array}{cc}
K_n(\lambda_i\lambda_j^*) & K_n(\lambda_ix_j^*) \\
K_n(x_i\lambda_j^*) & K_n(x_ix_j^*)
\end{array}
\right)\label{eqn6}\qquad \\
 =  n^{2k}\pi^{-2k}\det_{1\leq i, j\leq k}
\left(
\begin{array}{cc}
e^{-n(|\lambda_i|^2+|\lambda_j|^2)/2}K_n(\lambda_i\lambda_j^*) & e^{-n(|\lambda_i|^2+|x_j|^2)/2}K_n(\lambda_ix_j^*) \\
e^{-n(|x_i|^2+|\lambda_j|^2)/2}K_n(x_i\lambda_j^*) & e^{-n(|x_i|^2+|x_j|^2)/2}K_n(x_ix_j^*)
\end{array}
\right)
\end{multline}
Here the determinant is $(2k)\times(2k)$ with sub $2\times2$ blocks described above. We prove next that only the diagonal $2\times 2$ blocks can have a non-negligible contribution. Indeed consider the $(i, j)$-th $2\times2$ block. If $i\neq j$, then it has the estimation
\[
\left|
\begin{array}{cc}
e^{-n(|\lambda_i|^2+|\lambda_j|^2)/2}K_n(\lambda_i\lambda_j^*) & e^{-n(|\lambda_i|^2+|x_j|^2)/2}K_n(\lambda_ix_j^*) \\
e^{-n(|x_i|^2+|\lambda_j|^2)/2}K_n(x_i\lambda_j^*) & e^{-n(|x_i|^2+|x_j|^2)/2}K_n(x_ix_j^*)
\end{array}
\right|\sim
\left|
\begin{array}{cc}
o(1) & o(1) \\ o(1) & o(1)
\end{array}
\right|.
\]
This is because, for example consider the top-left entry, 
\[
e^{-n(|\lambda_i|^2+|\lambda_j|^2)/2}K_n(\lambda_i\lambda_j^*) \sim \exp(-n(|\lambda_i|^2 + |\lambda_j|^2 - 2\lambda_i^*\lambda_j)/2) = o(1)
\]
as $\lambda_i\neq\lambda_j$. Moreover the $o(1)$ notation decays to zero exponentially fast. The same is true for other terms. Hence the determinant of the small block above tends to zero.

However, if $i=j$, then
\[
\left|
\begin{array}{cc}
e^{-n|\lambda_i|^2}K_n(|\lambda_i|^2) & e^{-n(|\lambda_i|^2+|x_i|^2)/2}K_n(\lambda_ix_i^*) \\
e^{-n(|x_i|^2+|\lambda_i|^2)/2}K_n(x_i\lambda_i^*) & e^{-n|x_i|^2}K_n(|x_i|^2)
\end{array}
\right|\sim
1- e^{-n|\lambda_i-x_i|^2}.
\]
This is of order $\mathcal{O}(n^{-1/2})$ compared to the exponentially decay in the $i\neq j$ case.
As a consequence, in the expansion of the determinant in (\ref{eqn6}) over all permutations of $\mathcal{S}_{2k}$, only the terms consists of the entries in the diagonal $2\times2$ blocks can make a non-trivial contribution. Indeed, their contribution is exactly
\begin{equation}
\prod_{i=1}^k\int_{\lambda_i+B_n}\rho_2(\lambda_i, x_i)dx_i  \sim
 \prod_{i=1}^k\biggl(
n^2\pi^{-2}\int_{\lambda_i+B_n}(1-\exp(-n|\lambda_i-x_i|^2))dx_i
\biggl)^k\label{eqn7}
\end{equation}
Let $x_i = \lambda_i + n^{-3/4}u_i$, where $u_i\in B$, then
\[
\text{(\ref{eqn7})} = \biggl(n^{1/2}\pi^{-2}\int_B(1-\exp(-n^{-1/2}|u|^2))du\biggl)^k 
\sim\biggl(\pi^{-2}\int_{B}|u|^2du\biggl)^k.
\]

Having analyzed the $m=0$ case, next we prove that the contribution of the terms corresponding to $m\geq1$ is negligible.

By lemma \ref{lemma1} we have, for $m\geq1$,
\[
\rho_{2k+m}(\lambda_1, x_1, \ldots, \lambda_{k}, x_k, y_1, \ldots, y_m)\leq\rho_{2k}(\lambda_1, x_1, \ldots, \lambda_k, x_k)\prod_{i=1}^m\rho_1(y_i).
\]
Thus the contribution for $m\geq1$ is bounded by
\begin{equation}\label{v3.6}
\biggl(\int_{\lambda_1+B_n}dx_1\ldots\int_{\lambda_k+B_n}dx_k\rho_{2k}(\lambda_1, x_1, \ldots, \lambda_k, x_k)\biggl)\sum_{m\geq1}\frac{1}{m!}\biggl(\int_{b_n}\rho_1(y)dy\biggl)^m.
\end{equation}
Here $b_n = (\lambda_1+B_n)\sqcup\ldots\sqcup(\lambda_k+B_n)$ has size $\mathcal{O}(n^{-3/2})$. However since $|K_n(z)|\leq \exp(n|z|)$, we have
\[
|\rho_1(y)| = n\pi^{-1}\exp(-n|y|^2)|K_n(|y|^2)|\leq n\pi^{-1} = \mathcal{O}(n). 
\] Thus the second factor of (\ref{v3.6}) will tend to zero. Since the first factor is just the $m=0$ case, which converges as proved. Thus the whole expression converge to zero. Combining our result for the $m=0$ and the $m\geq1$ case, we successfully proved that $\widetilde{\rho}_k(\lambda_1, \ldots, \lambda_k) \to
(\frac{1}{\pi^2}\int_B|u|^2du)^k$ when $|\lambda_i|<1$.

As the second step, we prove that for distinct $\lambda_i$'s, if there exists some $i_0$ such that $|\lambda_{i_0}| > 1$, then $\widetilde{\rho}_k(\lambda_1, \ldots, \lambda_k) \to 0$. Indeed, we can assume that $\lambda_1, \lambda_2, \ldots, \lambda_p > 1+\epsilon_0$ and $\lambda_{p+1}, \ldots, \lambda_k < 1-\epsilon_0$ for some $\epsilon_0>0$ and $ p <k$. Then from the first part, we have shown in (\ref{v3.6}) that
\begin{multline}
\widetilde{\rho}_k(\lambda_1, \ldots, \lambda_k)\leq \biggl(\int_{\lambda_1+B_n}dx_1\ldots\int_{\lambda_k+B_n}dx_k\rho_{2k}(\lambda_1, x_1, \ldots, \lambda_k, x_k)\biggl)\times \\\sum_{m\geq0}\frac{1}{m!}\biggl(\int_{b_n}\rho_1(y)dy\biggl)^m.
\end{multline}
Again by Lemma \ref{lemma1} we have
\[
\rho_{2k}(\lambda_1, x_1, \ldots, \lambda_k, x_k)\leq\prod_{i=1}^k\rho_2(\lambda_i, x_i).
\]
Thus
\begin{equation}\label{v3.7}
\widetilde{\rho}_k(\lambda_1, \ldots, \lambda_k)\leq\prod_{i=1}^k\biggl(\int_{\lambda_i+B_n}\rho_2(\lambda_i, x)dx\biggl)\sum_{m\geq0}\frac{1}{m!}\biggl(\int_{b_n}\rho_1(y)dy\biggl)^m.
\end{equation}
We have already shown that for $i >p$, 
\[
\int_{\lambda_i + B_n}\rho_2(\lambda_i, x)dx \to \pi^{-2}\int_B|u|^2du
\]
which is finite. But for $i\leq p$, for $n$ sufficiently large, for all $x\in\lambda_i+B_n$, for $n$ sufficiently large we have $|x|^2\geq 1+\epsilon_0/2$. Then by Lemma \ref{good} we have
\[
|K_n(|x|^2)|\leq \mathcal{O}(1)\cdot e^{n|x|^2}\Bigl(|x|^2e^{-(|x|^2-1)}\Bigl)^n \leq  
\mathcal{O}(1)\cdot e^{n|x|^2}\Bigl((1+\epsilon_0/2)e^{-\epsilon_0/2}\Bigl)^n
\]
Hence for $x\in \lambda_i+B_n$,
\begin{eqnarray*}
\rho_2(\lambda_i, x) & = & n^2\pi^{-2}\exp(-n|\lambda_i|^2-n|x|^2)\Bigl( K_n(|\lambda_i|^2)K_n(|x|^2) - |K_n(\lambda_ix^*)|^2\Bigl)  \\
&\leq &n^2\pi^{-2}\exp(-n|\lambda_i|^2-n|x|^2)K_n(|\lambda_i|^2)K_n(|x|^2) \\
&\leq & \mathcal{O}(1)\cdot n^2\cdot \Bigl((1+\epsilon_0/2)e^{-\epsilon_0/2}\Bigl)^{2n} \to 0.
\end{eqnarray*}
Thus $\prod_{i=1}^k\int_{\lambda_i+B_n}\rho_2(\lambda_i, x)dx\to 0$ and the second factor in (\ref{v3.7}) is bounded as having been previously shown.
Hence we have $\widetilde{\rho}_k(\lambda_1, \ldots, \lambda_k) \to 0$. The proof is complete.
\end{proof}


\begin{lemma}\label{lemma4}
(Uniform boundedness) There is a constant $C$ depending only on the set $A$ such that, for any $n\geq1$ and $(\lambda_1, \ldots, \lambda_k)\in \Omega$, 
\[
\widetilde{\rho}_k(\lambda_1, \ldots, \lambda_k)<C.
\]
\end{lemma}

%
%

\begin{proof} In (\ref{v3.7}) we have shown that
\[
\widetilde{\rho}_k(\lambda_1, \ldots, \lambda_k)\leq\prod_{i=1}^k\biggl(\int_{\lambda_i+B_n}\rho_2(\lambda_i, x)dx\biggl)\sum_{m\geq0}\frac{1}{m!}\biggl(\int_{b_n}\rho_1(y)dy\biggl)^m.
\]
As is proved in Lemma \ref{lemma3},
the second factor does not depend on $\lambda$ and is convergent. 

We just need to prove that the first product is uniformly bounded. We note 
\begin{equation}\label{eqn13}
\rho_2(\lambda, x)  =  n^2\pi^{-2}
\det\left(
\begin{array}{cc}
e^{-n|\lambda|^2}K_n(|\lambda|^2) & e^{-n(|x|^2+|\lambda|^2)/2}K_n(\lambda x^*) \\
e^{-n(|x|^2+|\lambda|^2)/2}K_n(x \lambda^*) & e^{-n|x|^2}K_n(|x|^2)
\end{array}
\right).
\end{equation}

We consider the following three regions separately and shall prove that $\rho_2(\lambda, x)$ is uniformly bounded on all of them:
(1)$|\lambda| < 0.12,$ (2) $|\lambda| > 100$ and $(3) 0.12\leq|\lambda|\leq 100$.

%
%
\textsc{Part 1.} If $|\lambda| < 0.12$, then as $n$ is sufficiently large, $|\lambda|^2, |x|^2$ and $|\lambda x^*|$ are less than $0.02$. Writing $K_n(z)$ as $e^{nz}(1-R_n(z))$, by part (1) of Lemma \ref{good}, we get
\begin{eqnarray*}
\rho_2(\lambda, x) & = &  n^2\pi^{-2}
\det\left(
\begin{array}{cc}
1 & e^{-n(|x|^2+|\lambda|^2-2\lambda x^*)/2} \\
e^{-n(|x|^2+|\lambda|^2-2x\lambda^*)/2} & 1
\end{array}
\right) + o(1) \\
& = & n^2\pi^{-2}\Bigl(1-\exp(-n|x-\lambda|^2)\Bigl) +o(1)\\
& = & \mathcal{O}(n^2) \cdot \mathcal{O}(n^{-1/2}) + o(1) = \mathcal{O}(n^{3/2}).
\end{eqnarray*}
But the integration domain is $\mathcal{O}(n^{-3/2})$. Hence the integration of $\rho_2(\lambda, x)$ in $\lambda + B_n$ is uniformly bounded.

%
%
\textsc{Part 2.} If $|\lambda|>100$, then for sufficiently large $n$, we can ensure $|x|^2, |\lambda|^2, |x\lambda^*|>100$. By part (c) of the Lemma \ref{good}, this implies $|K_n(\lambda)^*|\leq C\cdot(100e^{-99})^n$ for some constants $C$. Hence
\[
n^2\pi^{-2}e^{-n(|x|^2-|\lambda|^2)/2}|K_n(\lambda x^*)| \leq Cn^2\pi^{-2}\cdot (100e^{-99})^n
\]
which is uniformly bounded above. The same is true for other terms in the determinant. Hence the integration of $\rho_2(\lambda, x)$ is bounded, just as in Case 1.

%
%
\textsc{Part 3.} Finally, we prove the most tricky part: $0.12\leq |\lambda|\leq 100$.
For $\rho_2(\lambda, x)$, we have
\begin{eqnarray}
\rho_2(\lambda, x)
& = & 
n^2\pi^2
\det\left(
\begin{array}{cc}
e^{-n|\lambda|^2}K_n(|\lambda|^2) & e^{-n\lambda x^*}K_n(\lambda x^*) \\
e^{-nx\lambda^*}K_n(x \lambda^*) & e^{-n|x|^2}K_n(|x|^2)
\end{array}
\right) + \nonumber\\
& & \qquad\qquad n^2\pi^2(e^{-2n\Re(x\lambda^*)}-e^{-n|\lambda|^2-n|x|^2})|K_n(\lambda x^*)|^2  \label{eq4}
\end{eqnarray}
First we analyze the second term in (\ref{eq4}). Using the inequality $|K_n(z)|\leq e^{n|z|}$ we have
\begin{eqnarray}
&&n^2\pi^{-2}|e^{-2n\Re(x\lambda^*)}-\exp({-n|\lambda|^2-n|x|^2})|\cdot|K_n(\lambda x^*)|^2 \nonumber\\
& \leq & n^2\pi^{-2}|e^{-2n\Re(x\lambda^*)}-\exp({-n|\lambda|^2-n|x|^2})|\cdot\exp({2n|\lambda x^*|}) \nonumber\\
& \leq & n^2\pi^{-2}\Bigl|\exp\Bigl(2n(|x\lambda^*| - \Re(x\lambda^*))\Bigl) - \exp\Bigl(-n(|\lambda|-|x|)^2\Bigl)\Bigl| \nonumber\\
\label{v3.8}&& \\
& \leq & n^2\pi^{-2}\Bigl|\exp\Bigl(2n(|x\lambda^*| - \Re(x\lambda^*))\Bigl)-1\Bigl| + n^2\pi^{-2}\Bigl|\exp\Bigl(-n(|\lambda|-|x|)^2\Bigl)-1\Bigl| \nonumber
\end{eqnarray}
Here (\ref{v3.8}) refers to the line next to it. For the first term of (\ref{v3.8}), we have $|x\lambda^*| - \Re(x\lambda^*) = \mathcal{O}(n^{-3/2})$ (see Remark \ref{remark1}). Thus the first term is of order $\mathcal{O}(n^2)\cdot \mathcal{O}(n^{-1/2}) = \mathcal{O}(n^{3/2})$.

Moreover, for the second term of (\ref{v3.8})
\[
n^2\pi^{-2}\Bigl(1-e^{-n(|\lambda|-|x|)^2}\Bigl) = \mathcal{O}(n^2)\cdot\Bigl(1-e^{-\mathcal{O}(n^{-1/2})}\Bigl) = \mathcal{O}(n^2)\cdot \mathcal{O}(n^{-1/2}) = \mathcal{O}(n^{3/2}).
\]
Hence (\ref{v3.8}) $=\mathcal{O}(n^{3/2})$ uniformly.
Since the integration region is of order $\mathcal{O}(n^{-3/2})$, we obtain that the integral of the second term in (\ref{eq4}) is $\mathcal{O}(n^{-3/2})\cdot \mathcal{O}(n^{3/2}) = \mathcal{O}(1)$ which is uniformly bounded.

Now it remains to prove that the first term in (\ref{eq4}) is $\mathcal{O}(n^{3/2})$ hence is uniformly bounded after integration on the region of order $\mathcal{O}(n^{-3/2})$.

Define $f(z) := e^{-nz}\sum_{\ell = 0}^{n-1}n^\ell z^\ell/\ell!$. Then the determinant is nothing but
\begin{equation}\label{eqn14}
\left|
\begin{array}{cc}
f(|\lambda|^2) & f(\lambda x^*) \\
f(x\lambda^*) & f(|x|^2) \\
\end{array}
\right|
= 
\left|
\begin{array}{cc}
f(|\lambda|^2) & f(\lambda x^*)-f(|\lambda|^2) \\
f(x\lambda^*)-f(|\lambda|^2) & f(|x|^2)-f(\lambda x^*)-f(x\lambda^*)+f(|\lambda|^2)  \\
\end{array}
\right|
\end{equation}
the equality holds because the determinant is invariant after subtracting the first row from the second one and then subtracting the first column from the second one. In order to prove $\rho_2(\lambda, x) = \mathcal{O}(n^{3/2})$, we just need to prove that the determinant in (\ref{eqn14}) is of order $\mathcal{O}(n^{-1/2})$. A simple calculation with Stirling's formula gives 
\begin{eqnarray*}
f'(z) & = & -ne^{-nz}\frac{n^nz^{n-1}}{n!} 
 =  -\sqrt{\frac{n}{2\pi}}e^{-\beta_n}e^{n(1-z)}z^{n-1},
\end{eqnarray*}
for some $\beta \in (\frac{1}{12n}, \frac{1}{12n+1})$.
we have 
\begin{equation}\label{eqn13}
f(x\lambda^*)-f(|\lambda|^2) = \int^{x\lambda^*}_{|\lambda|^2}f'(\xi)d\xi
\end{equation}
where the integration is along the straight line connecting the two end points. For $\xi$ lying between $|\lambda|$ and $x\lambda^*$, $\Im(\xi) = \mathcal{O}(n^{-3/4})$. 
\[
|f'(\xi)| \leq \mathcal{O}(1)\cdot \sqrt{n} \Bigl(e^{1-|\xi|}|\xi|\Bigl)^n \cdot e^{n(\Re(\xi) - |\xi|)} \leq \mathcal{O}(1)\cdot\sqrt{n}\cdot e^{n(\Re(\xi) - |\xi|)}.
\]

By Remark \ref{remark1}, $e^{n(\Re(\xi) - |\xi|)} = \mathcal{O}(1)$ if $\xi$ is uniformly bounded away from zero and $\Im(\xi) = \mathcal{O}(n^{-3/4})$. This can be achieved when $|\lambda| \geq 0.12$. Hence we have $f'(\xi) = \mathcal{O}(\sqrt{n})$ uniformly. However, in (\ref{eqn13}) the length of the integration region is  $|\lambda||x-\lambda| = \mathcal{O}(n^{-3/4})$. Thus we have $f(x\lambda^*)-f(|\lambda|^2) = \mathcal{O}(n^{-1/4})$. Similarly $f(\lambda x^*)-f(|\lambda|^2) = \mathcal{O}(n^{-1/4})$.
As a conclusion, the second term of the expansion of the determinant in (\ref{eqn14}) $(f(\lambda x^*) - f(|\lambda|^2))(f(x\lambda^*) - f(|\lambda|^2))$, is of order $\mathcal{O}(n^{-1/2})$. 

In order to prove that the first term in the expansion of (\ref{eqn14}) is also of order $\mathcal{O}(n^{-1/2})$, we just need to prove $f(|x|^2)-f(\lambda x^*)-f(x\lambda^*)+f(|\lambda|^2) = \mathcal{O}(n^{-1/2})$ because $f(|\lambda|^2)$ is bounded above.  Now
\begin{multline}\label{eqn15}
f(|x|^2)-f(\lambda x^*)-f(x\lambda^*)+f(|\lambda|^2) \\
 =  \int_\lambda^x (x^*f'(\xi x^*) - \lambda^*f'(\xi\lambda^*))d\xi = 
 \int_{\lambda}^{x} \int_{\lambda^*}^{x^*} \Bigl(f'(\xi\eta) + \xi\eta f''(\xi\eta)\Bigl)d\eta d\xi
\end{multline}
and 
\begin{equation}\label{v3.9}
f'(z) + zf''(z)  =  e^{-nz}\frac{n^nz^{n-1}}{n!}\cdot n^2(z-1)
 =  \frac{n^2}{e^{\beta_n}\sqrt{2\pi n}}z^{n-1}(z-1)e^{n(1-z)}
\end{equation}
where $\beta_n\in(\frac{1}{12n}, \frac{1}{12n+1})$.

Since the integration region of (\ref{eqn15}) is $\mathcal{O}(n^{-3/2})$, we just need to prove that $f'(z) + zf''(z) = \mathcal{O}(n)$ for $z$ such that $0.01\leq|z|\leq100$ and $\Im(z) = \mathcal{O}(n^{-3/4})$.

Thus by (\ref{v3.9}) the only thing left to do is to prove $(z-1)z^{n-1}e^{n(1-z)} = \mathcal{O}(n^{-1/2})$, or 
\begin{equation}\label{v3.10}
(z-1)^2z^{2n}e^{2n(1-z)} = \mathcal{O}(n^{-{1}}).
\end{equation}

Let $|z| = 1+u,$ and $z = (1+u)e^{i\theta}$. Then 
\begin{eqnarray*}
|(z-1)z^ne^{n(1-z)}|^2 & = & \mathcal{O}(1)\cdot |z-1|^2\Bigl((1+u)e^{-u}\Bigl)^{2n}
\end{eqnarray*}
where we used Remark \ref{remark1} again. 
Now since $\theta = \mathcal{O}(n^{-3/4})$ and $|u|<100$,
\begin{eqnarray*}
|z-1|^2 & = & u^2 + 2(1+u)(1-\cos\theta)\leq u^2 + (1+u)\theta^2\leq u^2 + 101\theta^2,
\end{eqnarray*}
thus $|z-1|^2 = u^2 + \mathcal{O}(n^{-3/2})$ and the constant in $\mathcal{O}(\cdot)$ is uniform in $z$. Hence 
\[
|(z-1)z^ne^{n(1-z)}|^2 \leq \mathcal{O}(1)\cdot u^2\Bigl((1+u)e^{-u}\Bigl)^{2n} + \mathcal{O}(n^{-3/2}).
\]
Let $g(u) = u\Bigl((1+u)e^{-u}\Bigl)^n$. Then it is easy to obtain that $g(u)$ achieves its maximum when $u = (1+\sqrt{1+4n})/(2n)= \mathcal{O}(n^{-1/2})$. The maximum value of $g(u) = \mathcal{O}(n^{-1/2})$. We obtain
\[
|(z-1)z^ne^{n(1-z)}|^2 \leq \mathcal{O}(1)\cdot g(u)^2 + \mathcal{O}(n^{-3/2}) \leq \mathcal{O}(n^{-1}).
\]
This is exactly (\ref{v3.10}). Thus we successfully proved that $f''(z) = \mathcal{O}(n)$, or, the determinant in $(\ref{eqn14})$ is of order $\mathcal{O}(n^{-1/2})$. We conclude that the integration of $\rho_2(\lambda, x)$ is uniformly bounded. 

From all the three parts above, $\widetilde{\rho}_k(\lambda_1, \ldots, \lambda_k)<C$ for some constant $C$ when $(\lambda_1, \ldots, \lambda_k)\in\Omega$. 
\end{proof}


\begin{lemma}\label{lemma5}
(Negligible set) Let $\overline{\Omega}$ be the complement of $\Omega$ in $I^k$. Then as $n\to\infty$
\[
\int_{\overline{\Omega}}\widetilde{\rho}_k(\lambda_1, \ldots, \lambda_k)d\lambda_1\ldots d\lambda_k\longrightarrow 0.
\]
\end{lemma}

%
%

\begin{proof} The strategy for this proof is to show that $|\overline{\Omega}|$ decays sufficiently fast --- much faster than the growth of $\widetilde{\rho}_k(\lambda_1, \ldots, \lambda_k)$. Indeed, if $(\lambda_1, \ldots, \lambda_k)\in\overline{\Omega}$, then this means there exists some clusters of at least three eigenvalues, which is a rare event. To formalize this idea, we need to introduce a rigorous definition of the cluster.

We define a equivalence relation in $\{\lambda_1, \ldots, \lambda_k\}$. For any $\lambda_i, \lambda_j$, if $|\lambda_i-\lambda_j|\in B_n$ then we say $\lambda_i\sim\lambda_j$. If there are a series of points $\lambda_{i_1}, \ldots, \lambda_{i_t}$ such that $\lambda_i\sim\lambda_{i_1}, \lambda_{i_s}\sim\lambda_{i_{s+1}}$ for $s = 1, 2, \ldots, t-1$ and $\lambda_{i_{t}}\sim\lambda_{j}$, then we also say $\lambda_i\sim\lambda_j$. In the end, we define $\lambda_i\sim\lambda_i.$ 

Suppose there are $p$ equivalent classes. From each of them we draw representatives $\lambda_{i_1}, \ldots, \lambda_{i_p}$, which is the largest number in that class with respect to the ordering $(\cdot\prec\cdot)$. 

Geometrically, each equivalent class represents a cluster of eigenvalues. For a more intuitive graph, please see Figure \ref{fig1}. In this figure from the $25$ eigenvalues on the complex plane,  we have 4 equivalent classes (clusters). In each of them we have a representative eigenvalue which is at the top of the cluster.  
\begin{figure}[!h]
    \centering
    \includegraphics[scale=.5]{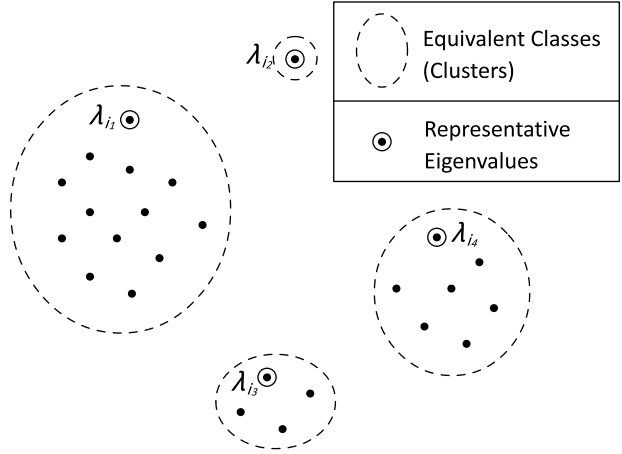}
    \caption{Equivalence Classes and Representative Eigenvalues. }
    \label{fig1}
\end{figure}


Then we have the obvious bound
\begin{eqnarray*}
\widetilde{\rho}_k(\lambda_1, \ldots, \lambda_k) & \leq & \int_{\lambda_{i_1}+B_n}dx_1\ldots\int_{\lambda_{i_p}+B_n}dx_p\rho_{k+p}(\lambda_1, \ldots, \lambda_k, x_1, \ldots, x_p)\\
& \leq & \prod_{j\neq i_1, \ldots, i_p}\rho_1(\lambda_j)\cdot\prod_{j=1}^p\int_{\lambda_{i_j}+B_n}\rho_2(\lambda_{i_j}, x_j)dx_j.
\end{eqnarray*}
where we used Lemma \ref{lemma1} in the second inequality.  As is proved in the Lemma \ref{lemma4}, the second product is uniformly bounded in $\lambda_{i_j}$ and $n$. Thus we only need to deal with the first product. 

But $\rho_1(\lambda_1) = \mathcal{O}(n)$. Thus $\widetilde{p}(\lambda_1, \ldots, \lambda_k) = \mathcal{O}(n^{k-p})$.  But the region of the integral is of $\mathcal{O}(n^{-3(k-p)/2})$, thus $\widetilde{\rho}_k$ is of order $\mathcal{O}(n^{-\frac{1}{2}(k-p)})$.  Since $k>p$, otherwise $(\lambda_1, \ldots, \lambda_k)$ will belong to $\Omega$, we conclude that 
\[
\int_{\overline{\Omega}}\widetilde{\rho}_k(\lambda_1, \ldots, \lambda_k)d\lambda_1\ldots d\lambda_k\longrightarrow 0. 
\]
\end{proof}

\textsc{Proof of Theorem \ref{thm1}}. With Lemma \ref{lemma3}, \ref{lemma4}, \ref{lemma5}, the proof of (\ref{eqn3}) is quite straightforward. 
\begin{multline*}
\int_{I^k}\widetilde{\rho}_k(\lambda_1, \dots, \lambda_k)d\lambda_1\ldots\lambda_k\\  =  \int_{I^k}\mathbbm{1}_{\Omega}\widetilde{\rho}_k(\lambda_1, \dots, \lambda_k)d\lambda_1\ldots\lambda_k + \int_{I^k}\mathbbm{1}_{\overline{\Omega}}\widetilde{\rho}_k(\lambda_1, \dots, \lambda_k)d\lambda_1\ldots d\lambda_k
\end{multline*}
The first term is uniformly bounded by Lemma \ref{lemma4} thus we can use the dominated convergence theorem. The second term converge to zero as stated by Lemma \ref{lemma5}. The proof is complete. \qquad $\square$ \\

\textsc{Proof of Corollary \ref{col1}}
We first denote $\widetilde{t}^{(n)}_1\leq\ldots\leq\widetilde{t}^{(n)}_k$ to be the $k$ smallest elements in the set $\{|\lambda_i-\lambda_{i^*}|: 1\leq i<n\}$. Correspondingly we define $\widetilde{\tau}^{(n)}_\ell = (\pi/4)^{1/4}\widetilde{t}^{(n)}_\ell$ for $1\leq \ell\leq k$. We first analyze the behavior of $\widetilde{\tau}^{(n)}_\ell$ instead. 

The event $\{x_\ell<n^{3/4}\widetilde{\tau}_\ell^{(n)}<y_\ell, 1\leq\ell\leq k\}$ is equivalent to
\begin{eqnarray*}
\chi^{(n)}((4/\pi)^{1/4}(x_k, y_k), \mathbb{C}) &\geq & 1,\\ 
\chi^{(n)}((4/\pi)^{1/4}(x_\ell, y_\ell), \mathbb{C}) &= & 1, 1\leq\ell\leq k-1,\\ 
\chi^{(n)}((4/\pi)^{1/4}(y_{\ell-1}, x_\ell), \mathbb{C}) &= & 0, 1\leq\ell\leq k. 
\end{eqnarray*}
Here $y_0=0$. The independence property of the Poisson process yields that
\begin{multline*}
\mathbb{P}(x_\ell<n^{3/4}\widetilde{\tau}_\ell^{(n)}<y_\ell, 1\leq\ell\leq k) \\ \longrightarrow
\biggl(1-e^{-(y_k^4-x_k^4)}\biggl)\prod_{\ell=1}^{k-1}(y_\ell^4-x_\ell^4)e^{-(y_\ell^4-x_\ell^4)}\prod_{\ell=1}^ke^{-(x_\ell^4-y_{\ell-1}^4)} 
=\biggl(e^{-x_k^4}-e^{-y_k^4}\biggl)\prod_{\ell=1}^{k-1}(y_\ell^4-x_\ell^4).
\end{multline*}
Let $x_\ell\to y_\ell-$, we can get that the density function for $n^{3/4}(\widetilde{\tau}_1^{(n)}, \ldots, \widetilde{\tau}_k^{(n)})$, in the limit, is proportional to
\[
u_1^3\ldots u_k^3e^{-u_k^4}.
\]
Thus, by integration w.r.t. $u_1, \ldots, u_{k-1}$ we can get the distribution function for $n^{3/4}\widetilde{\tau}_k$.
\begin{eqnarray*}
\mathbb{P}(\widetilde{\tau}_k\leq x)& =& c_k\int_0^xu_k^3e^{-u_k^4}du_k\int_{0<u_1<\ldots<u_k}u_1^3\ldots u_{k-1}^3du_1\ldots du_{k-1} \\
& = & c_k\int_0^x u_k^{4k-1}e^{-u_k^4}du_k\int_{0<v_1<\ldots<v_{k-1}<1}v_1^3\ldots v_{k-1}^3dv_1\ldots dv_{k-1}.
\end{eqnarray*}

We see that $\mathbb{P}(\widetilde{\tau}_\ell^{(n)}\in dx) = x^{4k-1}e^{-x^4}dx$. To prove that $\tau^{(n)}_\ell$ also has the same distribution, we just need to prove
\[
\mathbb{P}(\widetilde{t}^{(n)}_\ell\neq t^{(n)}_\ell \text{ for some }\ell \leq k)\to 0 \text{ as } n\to\infty.
\]

Recall that $t^{(n)}_\ell$ is the $\ell$-th smallest entry in the set $\{|\lambda_i-\lambda_j|: i\neq j\}$. Now we denote $|\lambda_{p_\ell} - \lambda_{q_{\ell}}| = t^{(n)}_\ell$. Then our first observation is that if there does not exist three eigenvalues $\lambda_{i_1}, \lambda_{i_2}, \lambda_{i_3}$ such that $|\lambda_{i_2} - \lambda_{i_1}|, |\lambda_{i_3} - \lambda_{i_1}| < 2\widetilde{t}^{(n)}_k$, then the $2k$ points $\{\lambda_{p_\ell}, \lambda_{q_\ell}: \ell = 1, \ldots, k\}$ must be distinct. This further implies that $t^{(n)}_\ell = \widetilde{t}^{(n)}_\ell$ for $\ell = 1, \ldots, k$. In short, we have
\[
\{ \widetilde{t}^{(n)}_\ell\neq t^{(n)}_\ell \text{ for some }\ell \leq k \} \Longrightarrow \{\exists\lambda_{i_1}, \lambda_{i_2}, \lambda_{i_3} \text{ such that }|\lambda_{i_2} - \lambda_{i_1}|, |\lambda_{i_3} - \lambda_{i_1}| < 2\widetilde{t}^{(n)}_k\}
\]
If we denote 
\[
\Xi^{(3)} = \sum_{\lambda_{i_1}, \lambda_{i_2}, \lambda_{i_3}\text{ are distinct}}\delta_{(\lambda_{i_1}, \lambda_{i_2}, \lambda_{i_3})}
\]
and
\[
\mathcal{B} = \{(\lambda, x_1, x_2): |\lambda|<2, |x_1-\lambda|, |x_2-\lambda|\leq Mn^{{-3/4}}\}
\]
where $M$ is a large constant, then we must have
\begin{eqnarray}
&&\mathbb{P}(\widetilde{t}^{(n)}_\ell\neq t^{(n)}_\ell \text{ for some }\ell \leq k) \nonumber\\& \leq & \mathbb{P}(\Xi^{(3)}(\mathcal{B})\neq0) + \mathbb{P}(\widetilde{t}^{(n)}_k> Mn^{-3/4}/2)+ \mathbb{P}(\text{there exists some } |\lambda_i|\geq2). \nonumber\\
&& \label{316}\\
& \leq & \mathbb{E}(\Xi^{(3)}(\mathcal{B})) + \mathbb{P}(\widetilde{t}^{(n)}_k> Mn^{-3/4}/2) + \mathbb{P}(\text{there exists some } |\lambda_i|\geq2). \nonumber
\end{eqnarray}
From the proof in Lemma \ref{lemma2} we know $\mathbb{E}(\Xi^{(3)}(\mathcal{B}))\to0$ as ${n\to\infty}$. From the circular law we also know that $\mathbb{P}(\text{there exists some } |\lambda_i|\geq2)\to0.$ Thus taking $n\to\infty$ in (\ref{316}) yields
\[
\limsup_{n\to\infty}\mathbb{P}(\widetilde{t}^{(n)}_\ell\neq t^{(n)}_\ell \text{ for some }\ell \leq k)\leq \mathbb{P}(n^{3/4}\widetilde{\tau}_k> (\pi/4)^{1/4} M/2)
\]
Finally taking $M\to\infty$ proves the result.

\section{Proof for the Wishart Ensemble Case}
Next we consider the complex Wishart ensemble. 
For $A_{mn}\sim W_2(m, n)$ being the Wishart ensemble, the joint distribution function of the eigenvalues of $A_{mn}$ is (see \cite{11}, \cite{12})
\[
p(\lambda_1, \ldots, \lambda_n) \propto \prod_{i<j}|\lambda_i-\lambda_j|^2\prod_{i=1}^n\lambda_i^{m-n} \exp\biggl(-m\sum_{i=1}^n\lambda_i\biggl).
\]
The $k$-point correlation function is given by 
\[
\rho_{k}(\lambda_1, \ldots, \lambda_k)  = \Bigl(K_n(\lambda_i, \lambda_j)\Bigl)_{i, j = 1}^k
\]
where
\begin{eqnarray}
&& K_n(x, y) \nonumber\\
& = & \sqrt{mn}\cdot\frac{\psi_{n-1}(mx)\psi_n(my)-\psi_n(mx)\psi_{n-1}(my)}{x-y}\label{4.-1} \\
& = & m^{3/2}n^{1/2}\biggl\{\psi_n(my)\int_0^1\psi'_{n-1}(mz_t)dt - \psi_{n-1}(my)\int_0^1\psi'_n(mz_t)dt\biggl\}\label{4.-2}
\end{eqnarray}
Here $z_t = tx + (1-t)y$ and
\[
\psi_\ell(x) = \sqrt{\frac{\ell!}{(\ell+m-n)!}}L_\ell^{(m-n)}(x)x^{(m-n)/2}e^{-x/2}
\]
and $L_{\ell}^{(m-n)}$ is the associated (generalized) Laguerre polynomial, i.e., 
\[
\int_0^\infty e^{-x}x^{m-n}L_p^{(m-n)}(x)L_q^{(m-n)}(x)dx = \frac{(p+m-n)!}{p!}\delta_{pq}.
\]

Now we denote $\beta := m/n$, we consider the case where $m, n\to\infty$ but their ratio $\beta\in[1, \infty)$. 

Note that in the Wishart ensemble case, thing are much simpler than the Ginibre case --- all the eigenvalues are real. We can use exactly the same scheme as in section 2 or in \cite{1}. The only difference is the kernel. In the following, we will argue Lemma \ref{lemma2} --- \ref{lemma5} still holds. For notational simplicity, we denote $a = (1-\beta^{-1/2})^2$ and $b = (1+{\beta^{-1/2}})^2$.

Again, for any $\epsilon>0$ we define 
\[
\xi^{(n)} := \sum_{j=1}^n\delta_{\lambda_i}\mathbbm{1}_{\lambda_i\in(a+\epsilon_0, b-\epsilon_0)} \text{ and }\widetilde{\xi}^{(n)} = \sum_{*}\delta_{\lambda_i}\mathbbm{1}_{\lambda_i\in(a+\epsilon_0, b-\epsilon_0)}
\]
where $\widetilde{\xi}^{(n)}$ is the thinning process by only keeping $\lambda_i$'s such that $\xi(\lambda_i + n^{-4/3}A) = 1$. To show $\chi^{(n)}(A\times I) = \widetilde{\xi}(I)$ asymptotically almost surely, the same argument in Lemma \ref{lemma2} still works. The only difference is to prove that for $\lambda\in(a+\epsilon_0, b-\epsilon_0)$ and $|\lambda-x|, |\lambda-y| = \mathcal{O}(n^{-4/3})$, 
\begin{equation}\label{n4.1}
\det\left(\begin{array}{ccc}K_n(\lambda, \lambda) & K_n(\lambda, x) & K_n(\lambda, y) \\
K_n(x, \lambda) & K_n(x, x) & K_n(x, y) \\ 
K_n(y, \lambda) & K_n(y, x) & K_n(y, y)\end{array}\right) = \mathcal{O}(n^{7/3}).
\end{equation}
Indeed, after subtracting the first column from the second and the third column, we obtain that the entries in the second and the third columns are $\mathcal{O}(n^{-4/3})\cdot |\partial_xK_n|$. If we can prove $|K_n| = \mathcal{O}(n)$ and $|\partial_xK_n| = \mathcal{O}(n^2)$, then each term in the expansion of the determinant of (\ref{n4.1}) is of order $\mathcal{O}(n^{7/3})$. The upper bound for $K_n$ and $\partial K_n$ is proved in Lemma \ref{lemma3.4}.

The point-wise convergence the correlation function, which is a similar version of Lemma \ref{lemma3}, is the most difficult part. In Lemma \ref{lemma3.3} we see that the diagonal $2\times2$ blocks can have nontrivial contributions, which is of order $\mathcal{O}(n^{4/3})$. Moreover we calculate this contribution explicitly. To prove that the contributions of the off-diagonal $2\times2$ blocks are negeligible, it remains to prove $|K_n| = \mathcal{O}(n), |\partial K_n| = \mathcal{O}(n^2),$  $|\partial^2K_n|= \mathcal{O}(n^3)$ when $|x-y| = \mathcal{O}(n^{-4/3})$ and $|K_n| = \mathcal{O}(1)$ when $|x-y|$ is bounded away from zero. This is because we can subtract the $(2j-1)$-th column from the $2j$-th column and then subtract $(2j-1)$-th row from the $(2j)$-th row, for all $j = 1, \ldots, k$. Then all the entry in the off-diagonal $2\times 2$ blocks are of order $\mathcal{O}(1)$ while the diagonal $2\times 2$ blocks are of order 
\begin{equation}\label{n44}
\left(\begin{array}{cc}|K_n| & |\partial_yK_n| \cdot\mathcal{O}(n^{-4/3}) \\
|\partial_xK_n|\cdot\mathcal{O}(n^{-4/3}) & |\partial_{xy}^2K_n|\cdot\mathcal{O}(n^{-8/3}) \end{array}\right) = \left(\begin{array}{cc}\mathcal{O}(n) & \mathcal{O}(n^{2/3}) \\
\mathcal{O}(n^{2/3}) & \mathcal{O}(n^{1/3}) \end{array}\right)
\end{equation}
Hence in the expansion of the $(2k)\times(2k)$ determinant, if one term involves off-diagonal terms, then it is of order $\mathcal{O}(n^{4k/3 - 1/3})$, which is negligible. We note that the upper bounds of $|\partial^jK_n|$ are shown in Lemma \ref{lemma3.4} for $j = 0, 1, 2$.

The uniformly boundedness result (Lemma \ref{lemma4}) and the negligible result (Lemma \ref{lemma5}) can also be obtained using exactly the same argument. The only difference is to prove that for $x, y\in (a+\epsilon_0, b-\epsilon_0)$ and $|x-y|\leq\mathcal{O}(n^{-4/3})$, 
\begin{equation}\label{nn4.2}
\det\left(\begin{array}{cc}K_n(x, x) & K_n(x, y) \\
K_n(y, x) & K_n(y, y) \end{array}\right) = \mathcal{O}(n^{-4/3}).
\end{equation}
Indeed, we use the same trick again. We subtract the first row from the second row, and then subtract the first column from the second one. Then (\ref{nn4.2}) exactly becomes (\ref{n44}). Hence the result holds.

In conclusion, Theorem \ref{thm1.3} holds if we can prove Lemma \ref{lemma3.3} and Lemma \ref{lemma3.4} below. Now we conquer  them one by one. 

\begin{lemma}\label{lemma3.3}
If $|x-y| = \mathcal{O}(n^{-4/3})$ and $x\in(a+\epsilon_0, b-\epsilon_0)$, then
\begin{multline}
\det\left(\begin{array}{cc}K_n(x, x) & K_n(x, y) \\
K_n(y, x) & K_n(y, y) \end{array}\right) \to \\ \frac{1}{3}\pi^2n^4u^2\Biggl[\frac{\beta}{2\pi}\cdot\frac{\sqrt{((1+{\beta^{-1/2}})^2-x)(x-(1-{\beta^{-1/2}})^2)}}{x}\Biggl]^4 + \mathcal{O}(n)
\end{multline}
where the $\mathcal{O}$ notation is uniform in $x, y$ and $u := y-x$. 
\end{lemma}

\begin{proof} 
The proof is quite lengthy and  just involves messy calculations.
By the Plancherel-Rotach asymptotic as described in \cite{2}, we have
\begin{multline}\label{4.0}
L_n^{(m-n)}(mx)  =  \frac{1}{x^{1/4}\sqrt{\pi n\sin\theta_0}}e^{mx/2}e^{-n(\beta-1)/2}\beta^{n/2}x^{-n(\beta-1)/2} \times \\
 \biggl\{\sin\Bigl[(n+\frac{1}{2})\theta_0 - \frac{n}{2}\sin2\theta_0+\frac{n\beta}{2}\sin2\theta_1-(n\beta+\frac{1}{2})\theta_1 + \frac{\pi}{4}\Bigl]+\mathcal{O}(n^{-1})\biggl\},
\end{multline}
where the angles $\theta_0, \theta_1$ are defined by 
\begin{eqnarray}
\cos\theta_0 & := & \frac{\beta-1-\beta x}{2\sqrt{\beta x}}, \qquad\theta_0\in(0, \pi) \label{4.1}\\
\sin\theta_1 & := & \sin\theta_0/\beta, \qquad\theta_1\in(0, \pi/2). \label{4.2}
\end{eqnarray}
By defining $\widehat{\beta}_n, \widehat{x}$ via
\begin{equation}\label{4.3}
n(\beta-1) = (n-1)(\widehat{\beta}-1), \qquad nx = (n-1)\widehat{x}
\end{equation}
we can obtain a similar formula for $L_{n-1}^{(m-n)}(mx)$, stated below. 
\begin{multline}\label{4.9}
L_{n-1}^{(m-n)}(mx)  =  \frac{1}{{x}^{1/4}\sqrt{\pi n\sin{\theta_0}}}e^{nx/2}e^{-n(\beta-1)/2}\widehat{\beta}^{n/2}\widehat{x}^{-n(\beta-1)/2} \times \\ 
\biggl\{\sin\Bigl[(n-\frac{1}{2})\theta_0 - \frac{n}{2}\sin2\theta_0+\frac{n\beta}{2}\sin2\theta_1-(n\beta-\frac{1}{2})\theta_1 + \frac{\pi}{4}\Bigl]+\mathcal{O}(n^{-1})\biggl\}.
\end{multline}

For notational simplicity, we define the angle inside the sine function of (\ref{4.0}), (\ref{4.9}) to be $M_+(\theta_{0, 1})$ and $M_-(\theta_{0, 1})$, respectively. Finally we can plug (\ref{4.0}) and (\ref{4.9}) in (\ref{4.-1}) to get
\begin{multline}\label{4.10}
K_n(x, y) = \frac{1}{(x-y)(xy)^{1/4}\pi\sqrt{\sin\theta_0\sin\phi_0}}\times \\
 = \biggl\{\sin M_-(\theta_{0, 1})\sin M_+(\phi_{0, 1}) - \sin M_+(\theta_{0, 1})\sin M_-(\phi_{0, 1}) + \mathcal{O}(n^{-1})\biggl\}.
\end{multline}
Here the angles $\phi_{0}, \phi_1$ are the counterparts of (\ref{4.1}), (\ref{4.2}) with respect to $y$. Moreover, using the same technique in \cite{3} by using higher order expansions, we can further show that the error term in (\ref{4.10}) is $\mathcal{O}(x-y)$ (this is intuitive because the error term is a symmetrical difference with respect to $x$ and $y$). 

By the equality 
\[
\sin(a-b)\sin(c+d)-\sin(c-d)\sin(a+b) = \sin(a+c)\sin(d-b)+\sin(a-c)\sin(b+d)
\]
we have
\begin{multline}\label{4.11}
K_n(x, y)  =  \frac{1}{(x-y)(xy)^{1/4}\pi\sqrt{\sin\theta_0\sin\phi_0}}\cdot \biggl\{ \mathcal{O}(x-y) -\sin\Bigl[\frac{1}{2}(\theta_0+\phi_0)-\frac{1}{2}(\theta_1+\phi_1)\Bigl]\times \\
\sin\Bigl[n(\theta_0-\phi_0)-\frac{n}{2}(\sin2\theta_0-\sin2\phi_0) + \frac{n\beta}{2}(\sin2\theta_1-\sin2\phi_1)-n\beta(\theta_1-\phi_1)\Bigl]
\biggl\}.
\end{multline}

Let $y = x+u$ where $u\sim \mathcal{O}(n^{-4/3})$. Then we can expand $\phi_0, \phi_1$ near $\theta_0, \theta_1$. We list the result below.
\begin{eqnarray*}
\phi_0-\theta_0 & = & \frac{\sqrt{\beta}\cos\theta_1}{2x\sin\theta_0}u + \mathcal{O}(u^2), \\
\phi_1-\theta_1 & = & \frac{\cos\theta_0}{2 x\sin\theta_0}u+\mathcal{O}(u^2).
\end{eqnarray*}
We then plug these approximations in (\ref{4.11}) to get
\begin{eqnarray*}
K_n(x, y) & = & \frac{1}{u(xy)^{1/4}\pi\sqrt{\sin\theta_0\sin\phi_0}}\cdot \biggl\{ \mathcal{O}(u) +\pi\sqrt{\beta}xg(x)
\sin\Bigl[\pi nug(x) +\mathcal{O}(nu^2)\Bigl]
\biggl\} \\
& = & \frac{1}{(xy)^{1/4}\pi\sqrt{\sin\theta_0\sin\phi_0}}\cdot \biggl\{ \mathcal{O}(1)+\pi\sqrt{\beta}xg(x)
\Bigl[\pi ng(x) - \frac{1}{6}\pi^3n^3u^2g^3(x) \Bigl]\biggl\}\\
&= & \frac{1}{\sqrt{x}\pi\sin\theta_0}\cdot \biggl\{ \mathcal{O}(1)+\pi\sqrt{\beta}xg(x)
\Bigl[\pi ng(x) - \frac{1}{6}\pi^3n^3u^2g^3(x) \Bigl]\biggl\} \\
& = & ng(x) - \frac{1}{6}\pi^2n^3u^2g^3(x) + \mathcal{O}(1). 
\end{eqnarray*}
where $g(x)$ is the density for Marc\'eko-Pastur Law
\[
g(x) := \frac{\beta}{2\pi}\cdot\frac{\sqrt{\big((1+{\beta^{-1/2}})^2-x\big)\big(x-(1-{\beta^{-1/2}})^2\big)}}{x}.
\]
Finally, by using the approximation, we can obtain the determinant
\begin{eqnarray*} 
K_n(x, x)K_n(y, y) - K_n(x, y)^2  =  \frac{1}{3}\pi^2g^4(x)n^4u^2 + \mathcal{O}(n).
\end{eqnarray*}
This completes the proof.  
\end{proof}

\begin{lemma}\label{lemma3.4}
Uniformly for $x, y\in(a+\epsilon_0, b-\epsilon_0)$, $K_n(x, y) = \mathcal{O}(n)$. Moreover, the first and second partial derivatives of $K_n(x, y)$ are of order $\mathcal{O}(n^2)$ and $\mathcal{O}(n^{3})$, respectively. Finally, if in addition $|x-y|>\delta$ for some constant $\delta$, then $K_n(x, y) = \mathcal{O}(1)$.
\end{lemma}

\begin{proof} By the Plancherel-Rotach asymptotics for Laguerre polynomials (\ref{4.1}) and the definition for $\psi_\ell(x)$, we obtain that $\psi_{n-k}(nx) = \mathcal{O}(n^{-1/2})$ for fixed integer $k$. By the formula
\[
\frac{d}{dx}\psi_\ell(x) = -\frac{1}{2}\psi_\ell(x) + \frac{m-n+2\ell}{2x}\psi_\ell(x) - \frac{\sqrt{\ell(\ell+m-n)}}{x}\psi_{\ell-1}(x),
\]
we obtain that $\psi'(n-k)(nx) = \mathcal{O}(n^{-1/2})$. Taking derivatives again yields $\psi''(n-k)(nx) = \mathcal{O}(n^{-1/2})$. If we plug this in (\ref{4.-2}), we observe that $K_n(x, y) = \mathcal{O}(n)$. Taking partial derivatives of $K_n(x, y)$ gives that the first order derivative of $K_n(x, y)$ is of order $\mathcal{O}(n^2)$ and the second order is of $\mathcal{O}(n^3)$.

On the other hand, if $|x-y|>\delta$ is uniformly bounded from zero, then from (\ref{4.-1}) we conclude $K_n(x, y) = \mathcal{O}(1)$. 
\end{proof}

\section{Proof for the Universal Unitary Ensemble Case}
For the UUE case, we proceed by using exactly the same argument as in the previous section. The $k$-point correlation function has a similar formula 
\[
\rho_k = \Bigl(K_n(\lambda_i, \lambda_j)\Bigl)_{i, j = 1}^n
\]
where the kernel $K_n(x, y)$ can also be defined via orthogonal polynomials. 
\[
K_n(x, y) = e^{-\frac{n}{2}(V(x) + V(y))}\sum_{j=0}^{n-1}p_j(x)p_j(y).
\]
Here $p_j(x)$ is the $j$-th orthonormal polynomial with respect to the weight $e^{-nV(x)}$.

Throughout the whole section we assume that the equilibrium measure $\Psi(x)dx$ is supported on a single interval $[a, b]$.

To prove Theorem \ref{thm1.4}, we need to run the argument again. As is analyzed in the previous section, we just need to prove the following two Lemmas. 

\begin{lemma}\label{lemma5.1}
For any $\epsilon_0>0$, if $|x-y|=\mathcal{O}(n^{-4/3})$ and $x\in(a+\epsilon_0, b-\epsilon_0)$, then
\[
\det\left(\begin{array}{cc}K_n(x, x) & K_n(x, y) \\ K_n(y, x) & K_n(y, y)\end{array}\right) \to \frac{\pi^2}{3}n^4u^2\Psi(x)^4 + \mathcal{O}(n),
\]
where the $\mathcal{O}$ notation is uniform in $x, y$ and $u := y-x$.
\end{lemma}

\begin{proof} We use Lemma 6.1 of \cite{15} which states that if $x\in(a+\epsilon_0, b -\epsilon_0)$, then
\[
\frac{1}{n\Psi(x)}K_n\biggl(x+\frac{\eta}{n\Psi(x)}, x+\frac{\xi}{n\Psi(x)}\biggl) = \frac{\sin\pi(\xi-\eta)}{\pi(\xi-\eta)} + \mathcal{O}(n^{-1})
\]
where the $\mathcal{O}$ notation is uniform in $x, \xi$ and $\eta$.
Thus we have
\begin{eqnarray*}
\left|\begin{array}{cc}K_n(x, x) & K_n(x, y) \\ K_n(y, x) & K_n(y, y)\end{array}\right|  & = & K_n(x, x)K_n(y, y) - K_n(x, y)K_n(y, x) \\
& = & n^2\Psi(x)^2 - n^2\Psi(x)^2\frac{\sin^2(\pi nu\Psi(x))}{\pi^2n^2u^2\Psi(x)^2} + \mathcal{O}(n) \\ 
& = & \frac{1}{3}\pi^2\Psi(x)^4n^4u^2 + \mathcal{O}(n).
\end{eqnarray*}
This completes the proof.
\end{proof}

\begin{lemma}\label{lemma5.2}
Uniformly for $x, y\in(a+\epsilon_0, b-\epsilon_0)$, $K_n(x, y) = \mathcal{O}(n)$. Moreover, the first and second partial derivatives of $K_n(x, y)$ are of order $\mathcal{O}(n^2)$ and $\mathcal{O}(n^{3})$, respectively. Finally, if in addition $|x-y|>\delta$ for some constant $\delta$, then $K_n(x, y) = \mathcal{O}(1)$.
\end{lemma}

For the proof of Lemma \ref{lemma5.2}, see Lemma 10.1 in \cite{4}.

In conclusion, we successfully extended the result to the UUE case. 

\section{Conclusion}
The minimum gap of the eigenvalues is the fine structure of the spectrum. This paper addresses two questions: (1) under what scale can we observe a clear picture of them? and (2) under the correct scale, what can we see? 

For the Ginibre ensemble where the eigenvalues lie on the two dimensional complex plane, the correct scale of the minimum gap is $\mathcal{O}(n^{-3/4})$. On the other hand for the Wishart ensemble and the universal unitary ensemble where we only have real eigenvalues, the correct scale is $\mathcal{O}(n^{-4/3})$. We also showed that all of the three cases have Poissonian limit after the correct scaling. This implies, heuristically, that the small eigenvalue spacing exhibit some asymptotic independency. 

Compared to Vinson's result in \cite{4} and Soshnikov's result in \cite{5}, we can obtain the joint distribution of $k$ smallest gaps, without requiring $k=1$. Our further research may include the generalization of the theory into general random point fields, or even in the high dimensional case.

\section*{Acknowledgment}
We would like to thank Professor Papanicolaou for precious advices and suggestions for this paper.


\begin{thebibliography}{9}

\bibitem{1} 
\textsc{Gerard Ben Arous, Paul Bourgade}, 
\textit{Extreme Gaps Between Eigenvalues of Random Matrices}, 
arXiv:1010.1294v1 [math.PR]

\bibitem{2} \textsc{A. P. Smith}, 
\textit{A New Asymptotic Form for the Laguerre Polynomials}, 
{J. Math. Phys.}, 33, 1666, 1992.

\bibitem{3} 
\textsc{B. Delyon, J. Yao}, \textit{On the spectral distribution of Gaussian random matrices}, 
Acta Mathematicae Sinica, English Series Vol. 22, No. 2, 297--312, 2006.

\bibitem{4} 
\textsc{J. Vinson}, 
\textit{Closest Spacing of Eigenvalues}, 
PhD thesis, Princeton University, 2001.

\bibitem{5} 
\textsc{A. Soshnikov}, 
\textit{Statistics of Entreme Spacings in Determinantal Random Point Processes}, 
Moscov Math. J., vol. 5, No. 3, 705--19, 2005.

\bibitem{6} 
\textsc{M. L. Mehta}, 
\textit{Random Matrices}, 
Second Edition, Academic Press, London, 1991.

\bibitem{7} \textsc{T. Tao, V. Vu, M. Krishnapur}, 
\textit{Random matrices: Universality of ESDs and the Circular Law}, 
Ann. Probab. Vol. 38, No. 5, 2023--2065, 2010.

\bibitem{8} \textsc{V.A. Marchenko, L.A. Pastur}, 
\textit{Distribution of eigenvalues for some sets of random matrices}, 
Mat. Sb. (N.S.), 72(114):4, 507--536, 1967.

\bibitem{9} 
\textsc{R. Horn, C. Johnson}, 
\textit{Matrix Analysis}, 
Cambridge University Press, 1985.

\bibitem{11} 
\textsc{A. Edelman}, 
\textit{Random Matrix Theory}, 
Acta Numerica, 1--65, 2005.

\bibitem{12} 
\textsc{U. Haagerup, S. Thorbj\o rnsen}, 
\textit{Random Matrices with Complex Gaussian Entries}, 
Expo. Math, Vol. 21, 293--337, 1998.

\bibitem{13} \textsc{J. P. Keating}, 
\textit{Random Matrices and the Riemann Zeta-Function: a Review}, 
Applied Mathematics Entering the 21st Century: Invited
Talks from the ICIAM 2003 Congress, 210--225, 2004.

\bibitem{14} \textsc{N. M. Katz, P. Sarnack}, 
\textit{Random Matrices, Frobenius Eigenvalues and Monodromy}, 
American Mathematics Society Colloquium Publications, 45. American Mathematical Society, Providence, Rhode island, 1999.

\bibitem{15} 
\textsc{P. Deift, T. Keiecherbauer, K. T-R McLaughlin, S. Venakides and X. Zhou}, 
\textit{Strong Asymptotics of Orthogonal Polynomials with respect to Exponential Weights}, 
Comm. Pure Appl. Math., 52(12): 1491--1552, 1999.

\bibitem{16} 
\textsc{A. Soshnikov}, 
\textit{Level Spacings Distribution for Large Random Matrices: Gaussian Fluctuations}, 
Ann. of Math., Vol. 5, No. 3, 705--719, 2005.

\bibitem{17} 
\textsc{T. Tao, V. Vu}, 
\textit{Random MatricesL Universality of Local Eigenvalues Statistics}, 
Acta Math., to appear.

\bibitem{18} 
\textsc{L. Erd\"{o}s, P\'{e}ch\'{e}, J. Ram\'{i}rez, B. Schlein and H.T. Yau}, 
\textit{Bulk Universality for Wigner Matrices}, 
Comm. Pure Appl. Math., Vol. 63, Issue 7, 895--925, 2010.

\bibitem{19}
\textsc{L. Erd\"{o}s, J. Ram\'{i}rez, B. Schlein, T. Tao, V. Vu, and H.T. Yau}, 
\textit{Bulk Universality for Wigner Hermitian Matrices with Subexponential Decay}, 
Math. Res. Lett. 17, No. 04, 667--674, 2010.
\end{thebibliography}
\end{document}